\documentclass[10pt]{article}
\usepackage{amssymb,enumerate}
\usepackage{amsmath,amsfonts}
\usepackage{amsthm}
\usepackage{latexsym}
\usepackage{mathrsfs}
\usepackage{color}
\usepackage[a4paper]{geometry}
\usepackage{algorithm}
\usepackage{algorithmicx}
\usepackage{algpseudocode}
\usepackage{microtype}

\DeclareMathOperator{\cK}{\ensuremath{\mathcal{K}}}
\DeclareMathOperator{\cN}{\ensuremath{\mathcal{N}}}
\DeclareMathOperator{\cC}{\ensuremath{\mathcal{C}}}
\DeclareMathOperator{\cO}{\ensuremath{\mathcal{O}}}
\DeclareMathOperator{\cS}{\ensuremath{\mathcal{S}}}

\DeclareMathOperator{\st}{\ensuremath{\mathrm{s.t.}}}
\DeclareMathOperator{\bR}{\ensuremath{\mathbb{R}}}

\DeclareMathOperator{\rmint}{\ensuremath{\mathrm{int}}}

\DeclareMathOperator{\sgn}{\ensuremath{\mathrm{sgn}}}
\newtheorem{lemma}{Lemma}

\newtheorem{assumption}{Assumption}
\newtheorem{definition}{Definition}
\newtheorem{theorem}{Theorem}
\newtheorem{remark}{Remark}
\newtheorem{proposition}{Proposition}
\newcommand{\beq}{\begin{equation}}
\newcommand{\eeq}{\end{equation}}
\newcommand{\beqa}{\begin{eqnarray}}
\newcommand{\eeqa}{\end{eqnarray}}
\newcommand{\beqas}{\begin{eqnarray*}}
\newcommand{\eeqas}{\end{eqnarray*}}
\newcommand{\ba}{\begin{array}}
\newcommand{\ea}{\end{array}}
\newcommand{\bi}{\begin{itemize}}
\newcommand{\ei}{\end{itemize}}

\def\bbK{{\mathbb K}}
\def\bd{{\bar d}}

\def\bmu{{\bar \mu}}

\def\bX{{\bar X}}

\def\cK{{\cal K}}
\def\cS{{\cal S}}

\def\Diag{{\rm Diag}}
\def\hd{{\widehat d}}
\def\td{{\widetilde d}}
\def\tr{{\widetilde r}}
\def\tmu{{\widetilde \mu}}

\allowdisplaybreaks
\title{A Newton-CG based barrier method for finding a second-order stationary point of nonconvex conic optimization with complexity guarantees}

\author{
Chuan He
\thanks{
Department of Industrial and Systems Engineering, University of Minnesota, USA (email: {\tt he000233@umn.edu}, {\tt zhaosong@umn.edu}). The work of the second author was partially supported by NSF Award IIS-2211491.}
\and
Zhaosong Lu
\footnotemark[1]
}

\date{November 2, 2021 (Revised: June 23, 2022; September 30, 2022)}

\begin{document}
\maketitle

\begin{abstract}
In this paper we consider finding an approximate second-order stationary point (SOSP) of  nonconvex conic optimization that minimizes a twice differentiable function over the intersection of an affine subspace and a convex cone. In particular, we propose a Newton-conjugate gradient (Newton-CG) based barrier method for finding an $(\epsilon,\sqrt{\epsilon})$-SOSP of this problem. Our  method is not only implementable, but also achieves an iteration complexity of $\cO(\epsilon^{-3/2})$, which matches the best known iteration complexity of second-order methods for finding an $(\epsilon,\sqrt{\epsilon})$-SOSP of unconstrained nonconvex optimization. The operation complexity, consisting of $\cO(\epsilon^{-3/2})$ Cholesky factorizations and $\widetilde{\cO}(\epsilon^{-3/2}\min\{n,\epsilon^{-1/4}\})$ other fundamental operations, is also established for our method.\footnote{The number $n$ is the problem dimension and $\widetilde{\cO}(\cdot)$ represents $\cO(\cdot)$ with logarithmic terms omitted.}
\end{abstract}


\bigskip


\section{Introduction} \label{intro}
In this paper we consider the conic constrained optimization problem:
\begin{equation}\label{conic-prob}
\min_x \{f(x):Ax=b, x\in\mathcal{K}\},
\end{equation}
where $A\in\bR^{m\times n}$ is of full row rank, $b\in\bR^m$, and $\cK\subseteq\bR^n$ is a closed and pointed convex cone with nonempty interior. Assume that problem \eqref{conic-prob} has at least an optimal solution. In addition, assume that Slater's condition holds for this problem,  i.e., $\Omega^{\rm o} = \{x: Ax=b, x\in\rmint \cK\} \neq\emptyset$,  and $f$ is twice continuously differentiable and nonconvex on $\Omega^{\rm o}$, where $\rmint \cK$ denotes the interior of $\cK$.

In recent years there have been numerous developments on algorithms with complexity guarantees for finding an approximate second-order stationary point (SOSP) of some special cases of problem \eqref{conic-prob}. 
In particular, cubic regularized Newton methods \cite{AABHM17,CGT11ARC,NP06cubic}, trust-region methods \cite{CRRW21trust,CRS16trust,MR17trust},  quadratic regularization method  \cite{BM17QC}, accelerated gradient-type method \cite{CDHS17},  second-order line-search method \cite{RW18}, inexact regularized Newton method \cite{CRS19iN}, and Newton-CG method \cite{ROW20} were proposed for finding an approximate SOSP of a special case of \eqref{conic-prob} with $A=0$, $b=0$ and $\cK=\bR^n$, that is, an unconstrained smooth optimization problem 
\beq \label{unconstr-opt}
\min_x f(x),
\eeq
where $\nabla^2 f$ is assumed to be Lipschitz continuous in a certain level set of $f$.  These methods enjoy an iteration complexity of $\cO(\epsilon^{-3/2})$
for finding an $(\epsilon,\sqrt{\epsilon})$-SOSP $x$ of \eqref{unconstr-opt}  that satisfies
\[
\|\nabla f(x)\|\le\epsilon, \quad \lambda_{\min}(\nabla^2 f(x))\ge-\sqrt{\epsilon},
\]
where $\epsilon\in(0,1)$ is a tolerance parameter, and $\lambda_{\min}(\cdot)$ denotes the minimum eigenvalue of the associated matrix. This iteration complexity is proved to be optimal in \cite{CDHS19lower,CGT18worst}. In addition to iteration complexity, the operation complexity of the methods  \cite{AABHM17,CDHS17,CRRW21trust,ROW20,RW18} was also studied, which is  measured by the amount of {\it fundamental operations} consisting of gradient evaluations and Hessian-vector products of $f$. Under some suitable assumptions, it was shown  that these methods have an operation complexity of
$\widetilde{\cO}(\epsilon^{-7/4})$ for finding an $(\epsilon,\sqrt{\epsilon})$-SOSP of  \eqref{unconstr-opt} with high probability. Similar operation complexity bounds are also achieved by some gradient-based algorithms with  random perturbations
(e.g., see \cite{AL17,JNJ18AGD,XJY17AGD}).

Recently, a log-barrier Newton-conjugate gradient (Newton-CG)  method  was proposed in \cite{OW21} for finding an approximate SOSP of a special case of \eqref{conic-prob} with $A=0$, $b=0$ and $\cK=\bR^n_+$, namely, the problem
\beq \label{boxconstr-prob}
\min_x \{f(x):x\ge0\},
\eeq 
where $\nabla^2 f$ is assumed to be Lipschitz continuous in a certain subset of the interior of $\bR^n_+$. 
Instead of solving \eqref{boxconstr-prob} directly, this method applies a preconditioned Newton-CG method, which is a variant of Newton-CG method \cite{ROW20},  to minimize a log-barrier function associated with \eqref{boxconstr-prob}.
Under some suitable assumptions, it was shown in \cite{OW21} that this method has an iteration complexity of ${\cO}(\epsilon^{-3/2})$ and an operation complexity of $\widetilde{\cO}(\epsilon^{-7/4})$
 for finding an $(\epsilon,\sqrt{\epsilon})$-SOSP $x$ of \eqref{boxconstr-prob} that satisfies
\beq \label{2nd-order-pt}
x>0, \quad \nabla f(x) \ge -\epsilon e, \quad \|\bX\nabla f(x)\|_\infty \le \epsilon, \quad \lambda_{\min}(\bX\nabla^2f(x)\bX) \ge -\sqrt{\epsilon}
\eeq
with high probability, where $e$ is the all-ones vector, and $\bX$ is a diagonal matrix whose $i$th diagonal entry is $\min\{x_i,1\}$.  
 Besides, the earlier work \cite{BCY2015} proposed an interior-point method with an iteration complexity of $\cO(\epsilon^{-3/2})$ for finding a point $x$ satisfying the first, third and last relations in \eqref{2nd-order-pt} with $\bar{X}$ being replaced by $X=\Diag(x)$, where $\Diag(x)$ is a diagonal matrix with $x$ on its diagonal. This method solves a preconditioned second-order trust-region subproblem per iteration. More recently, a projected Newton-CG method with complexity guarantees was proposed in \cite{XW21pN} for finding an approximate SOSP of a more general form of \eqref{boxconstr-prob} with only a subvector of $x$ being nonnegative.

In addition, an interior-point method was proposed in \cite{HLY19} for finding an approximate SOSP of a special case of \eqref{conic-prob} with $\cK=\bR^n_+$, that is, {a linearly constrained smooth optimization problem}
\beq \label{lpconstr-prob}
\min_x \{f(x):Ax=b,\ x\ge0\}.
\eeq 
This method solves a preconditioned second-order trust-region subproblem per iteration, which minimizes a possibly nonconvex quadratic function over the intersection of a linear subspace and an Euclidean ball. Under some suitable assumptions, it was shown  in \cite{HLY19} that this method has an iteration complexity of $\cO(\epsilon^{-3/2})$ for finding an 
{$(\epsilon,\sqrt{\epsilon})$-SOSP}
$x$ of \eqref{lpconstr-prob} that satisfies
\[
\ba{l}
Ax=b,\ x>0, \ \nabla f(x)+A^T\lambda \ge -\epsilon e, \ \|X(\nabla f(x)+A^T\lambda)\|_\infty \le \epsilon, \\ [8pt]
d^T(X\nabla^2f(x)X+\sqrt{\epsilon} I)d \ge 0 \quad \forall  d \in \{d:AXd=0\} 
\ea
\]
for some $\lambda\in\bR^m$. It is worth mentioning that this method requires solving the associated trust-region subproblems {\it exactly}, which is typically an impossible task. Thus, this method is not implementable in general.

Besides, several methods including trust-region methods \cite{BSS87,CLY02}, sequential quadratic programming method \cite{BL95}, two-phase method \cite{CGT19eq,CM19poly}, penalty method \cite{GEB22cpt}, and augmented Lagrangian (AL) type methods \cite{AHRS17,BHR18AL,HLP22,S19iAL,XW21PAL} were developed for finding an SOSP of nonconvex equality constrained optimization. In addition, a projected gradient descent method with random perturbations was proposed in \cite{LRYHH20} for nonconvex optimization with linear inequality constraints.

The aforementioned methods are not suitable for finding an approximate SOSP of problem \eqref{conic-prob} in general. On the other hand, in the concurrent work \cite{DS21HBA}, the authors proposed a Hessian barrier algorithm and studied its iteration complexity for finding an approximate SOSP of problem~\eqref{conic-prob}. This algorithm nicely generalizes the cubic regularized Newton method \cite{NP06cubic} to problem \eqref{conic-prob}. However, it requires solving many cubic regularized projected Newton subproblems {\it exactly}, which is typically impossible to implement. To the best of our knowledge, there is yet no implementable method with complexity guarantees in the literature for finding an approximate SOSP of problem \eqref{conic-prob}.

Inspired by \cite{BCY2015,HLY19,OW21,ROW20}, in this paper we develop an implementable method with complexity guarantees for finding an approximate SOSP of problem \eqref{conic-prob}. Our main contributions are as follows. 
\bi
\item We introduce a novel notion of an approximate SOSP of \eqref{conic-prob}, by the use of the self-concordant barrier function associated with the cone $\cK$ and the study of optimality conditions of \eqref{conic-prob}.
\item We propose an implementable Newton-CG based barrier method for finding an approximate SOSP of \eqref{conic-prob}, whose main operations consist of Cholesky factorizations and other fundamental operations including Hessian-vector products of $f$, matrix multiplications,  and backward or forward substitutions to a triangular linear system.
This method generalizes the log-barrier Newton-CG method \cite{OW21} proposed for \eqref{boxconstr-prob} to the optimization problems with affine and general conic constraints, and thus provides an affirmative answer to the open question raised by O'Neill and Wright at the end of \cite{OW21}.
\item We show that under mild assumptions, the proposed method achieves an iteration complexity of $\cO(\epsilon^{-3/2})$ and also an operation complexity, consisting of $\cO(\epsilon^{-3/2})$ Cholesky factorizations and $\widetilde{\cO}(\epsilon^{-3/2}\min\{n,\epsilon^{-1/4}\})$ other fundamental operations mentioned above, for finding an $(\epsilon,\sqrt{\epsilon})$-SOSP of \eqref{conic-prob} with high probability. When $\cK$ is the nonnegative orthant, these complexity results match the best known ones for finding an $(\epsilon,\sqrt{\epsilon})$-SOSP of \eqref{unconstr-opt} or \eqref{boxconstr-prob} with high probability (e.g., see \cite{CRRW21trust,OW21,ROW20}).
\item The complexity results of our method are established under the assumption that $\nabla^2 f$ is  \textit{locally Lipschitz continuous} in a certain subset of $\Omega^{\rm o}$ (see Assumption~\ref{main-assump}(b)).  Such an assumption is weaker than the one based on the global Lipschitz continuity of $\nabla^2 f$ usually imposed in the literature (e.g., see \cite{OW21}). As a consequence, our method is applicable to the problems with a broader class of objective functions $f$ (see Section \ref{sec:complexity} for more discussion).
\ei

The rest of this paper is organized as follows. In Section \ref{sec:preliminary}, we introduce some notation and study some properties of logarithmically homogeneous self-concordant barrier functions. In  Section \ref{sec:approx-opt-cond}, we study optimality conditions of problem \eqref{conic-prob} and introduce an approximate counterpart of them. In Section \ref{sec:algorithm}, we propose a Newton-CG based barrier method. Finally, we establish iteration and operation complexity results for the proposed method in Section \ref{sec:complexity}.

\section{Notation and preliminaries} \label{sec:preliminary} 

In this section we introduce some notation and also study some properties of a logarithmically homogeneous self-concordant barrier function for a closed convex cone that will be used in this paper. 

Throughout this paper, let $\bR^n$ denote the $n$-dimensional Euclidean space and $\langle\cdot,\cdot\rangle$ denote the standard inner product. 
We use $\|\cdot\|$ to denote the Euclidean norm of a vector or the spectral norm of a matrix.
We denote by $\lambda_{\min}(H)$ the minimum eigenvalue of a real symmetric matrix $H$. For any two real symmetric matrices $M_1$ and $M_2$, $M_1 \preceq M_2$ means that $M_2-M_1$ is positive semidefinite. For any positive semidefinite matrix $M$, $M^{1/2}$ denotes a positive semidefinite matrix such that $M=M^{1/2}M^{1/2}$. For the closed convex cone $\cK$, its interior and dual cone are  denoted by $\rmint\cK$ and $\cK^*$, respectively. For any $x\in\cK$, the normal cone of $\cK$ at $x$ is denoted by $\mathcal{N}_{\cK}(x)$.  For any $t\in \mathbb{R}$, we let ${\rm sgn}(t)$ be $1$ if $s \ge 0$ and let it be $-1$ otherwise. In addition, we use order notation $\mathcal{O}(\cdot)$ in its usual sense, and notation $\widetilde{\mathcal{O}}(\cdot)$ to represent the order with hidden logarithmic factors.

Logarithmically homogeneous self-concordant (LHSC) barrier functions have played a crucial role in the development of interior point methods for solving convex conic programming (see the monograph \cite{NN94}). The design and analysis of the Newton-CG based barrier method in this paper also heavily rely on an  LHSC barrier function. Throughout this paper, we assume that the cone $\cK$ is equipped with a \emph{$\vartheta$-logarithmically homogeneous 
self-concordant} ($\vartheta$-LHSC) barrier function $B$ for some $\vartheta \ge 1$.
That is, $B: \rmint \cK \to \bR$ satisfies the following conditions:
\begin{enumerate}[{\rm (i)}]
\item $B$ is  convex and three times continuously differentiable in $\rmint\cK$,
 and moreover, $|\varphi^{\prime\prime\prime}(0)|\le 2(\varphi^{\prime\prime}(0))^{3/2}$ holds for all $x\in\rmint\cK$ and $u\in\bR^n$, where $\varphi(t)=B(x+tu)$;
\item $B$ is a \emph{barrier function} for $\cK$, that is, $B(x)$ goes to infinity as $x$ approaches the boundary of $\cK$;
\item $B$ satisfies the \emph{logarithmically homogeneous property}:
\begin{equation}\label{def:log-hm-self-concordant-barrier}
B(tx) = B(x)-\vartheta \ln t\ \quad\forall x\in\rmint \cK, t>0.
\end{equation}
\end{enumerate}
For the details of LHSC barrier function and its examples, we refer the reader to \cite{NN94} and the references therein.

For any $x\in\rmint\cK$, the function $B$ induces the following so-called local norms:
\begin{eqnarray}
\|v\|_x&:=& \left(v^T\nabla^2 B(x)v\right)^{1/2}\ \quad \forall v\in\bR^n, \nonumber \\ [4pt]
\|v\|_x^*&:=& \left(v^T[\nabla^2 B(x)]^{-1}v\right)^{1/2}\ \quad \forall v\in\bR^n, \nonumber \\ [4pt]
\|M\|^*_x &:=& \max\limits_{\|v\|_x \le 1} \|Mv\|^*_x\ \quad \forall M\in \bR^{n\times n}. \label{M-norm}
\end{eqnarray}

In the remainder of this section, we study some properties of the $\vartheta$-LHSC barrier function $B$ that will be used subsequently in this paper.

\begin{lemma}\label{lem:barrier-property}
Let $x\in\rmint\cK$ and $\beta \in (0,1)$ be given. Then the following statements hold for the $\vartheta$-LHSC barrier function $B$.
\begin{enumerate}[{\rm (i)}]
\item $(\|\nabla B(x)\|_x^*)^2=-x^T\nabla B(x)=\|x\|_x^2=\vartheta$.
\item $-\nabla B(x)\in\rmint\cK^*$.
\item $\{y:\|y-x\|_x< 1\}\subset \rmint\cK$.
\item For any $y$ satisfying $\|y-x\|_x\le \beta$, it holds that
\begin{equation}
(1-\beta)\|v\|_x^*\le \|v\|_y^* \le (1-\beta)^{-1}\|v\|_x^*\ \quad \forall v\in\bR^n.\label{ineq:dual-local-norm-iterate-ppty}
\end{equation}
\item $B(x+d)\le B(x)+\nabla B(x)^Td+\frac{1}{2}d^T\nabla^2B(x)d+\frac{1}{3(1-\beta)}\|d\|_x^3 \ \text{whenever} \  \|d\|_x\le\beta$.
\item $\{s: \|s+\nabla B(x)\|_x^*\le 1\} \subseteq \cK^*$.
\end{enumerate}
\end{lemma}

\begin{proof}
The proof of statements (i), (ii), and (iii) can be found in \cite[Proposition~2.3.4]{NN94}, \cite[Theorem~2.4.2]{NN94}, and \cite[Theorem~2.1.1]{NN94}, respectively.

We now prove statement (iv). Let $y$ be such that $\|y-x\|_x\le \beta$. It follows from \cite[Theorem 2.2.1]{NN94} that  
\[
(1-\beta)^2\nabla^2 B(x)\preceq \nabla^2 B(y)\preceq (1-\beta)^{-2}\nabla^2 B(x),
\]
which, together with the positive definiteness of $\nabla^2 B(x)$ and $\nabla^2 B(y)$, implies that 
\[
(1-\beta)^2[\nabla^2 B(x)]^{-1}\preceq [\nabla^2 B(y)]^{-1}\preceq (1-\beta)^{-2}[\nabla^2 B(x)]^{-1}.
\]
Statement (iv) then immediately follows from these relations.

We next prove statement (v). Let $d\in\bR^n$ be such that $\|d\|_x\le\beta$. By \cite[Theorem 4.1.8]{Nes03},  one has 
\beq \label{uppbnd-B}
B(x+d)\le B(x) +\nabla B(x)^Td + \rho(\|d\|_x),
\eeq
where $\rho(t)=-\ln(1-t)-t$. Notice that $\rho(t)=\sum^\infty_{k=2} t^k/k$ for each $t\in(0,1)$, and $\|d\|^2_x=d^T\nabla^2 B(x)d$. Using these, $\|d\|_x\le\beta<1$ and \eqref{uppbnd-B}, we obtain that
\[
\begin{array}{rcl}
B(x+d)&\le&B(x) +\nabla B(x)^Td + \frac{1}{2} d^T\nabla^2 B(x)d + \sum_{k=3}^\infty\frac{\|d\|_x^k}{k}\\[5pt]
&\le&B(x) +\nabla B(x)^Td + \frac{1}{2} d^T\nabla^2 B(x)d + \frac{\|d\|_x^3}{3} \sum_{k=3}^\infty\beta^{k-3}\\[5pt]
&=&B(x) +\nabla B(x)^Td + \frac{1}{2} d^T\nabla^2 B(x)d + \frac{\|d\|_x^3}{3(1-\beta)}.
\end{array}
\]
Hence, statement (v) holds as desired.

We finally prove statement (vi). By \cite[Theorem 2.4.1]{NN94}, we know that $B^*$ is a $\vartheta$-LHSC barrier function for the cone $-\cK^*$, where $B^*$ is the conjugate of $B$ defined as
\[
B^*(y)=\sup_{x\in\rmint\cK}\{\langle y,x\rangle-B(x)\}\ \quad \forall y \in \rmint(-\cK^*).
\]
For any $y\in\rmint(-\cK^*)$, let $\|\cdot\|'_y$ be the local norm induced by $B^*$, that is, $\|s\|'_y=\sqrt{s^T\nabla^2 B^*(y)s}$ for any $s\in\bR^n$. Since $x\in\rmint\cK$, notice from statement (ii) that $\nabla B(x)\in\rmint(-\cK^*)$. Also, from the proof \cite[Theorem 2.4.2]{NN94},  one has $\nabla^2 B^*(\nabla B(x))=[\nabla^2 B(x)]^{-1}$. It then follows that 
\[
\|s\|'_{\nabla B(x)}=\sqrt{s^T\nabla^2 B^*(\nabla B(x))s}=\sqrt{s^T[\nabla^2 B(x)]^{-1}s}=\|s\|^*_x\ \quad \forall s\in\bR^n.
\]
In view of this and statement (iii) with $\cK$ and $x$ replaced respectively by $-\cK^*$ and $\nabla B(x)$, one has that
\[
\{s:\|s-\nabla B(x)\|_x^* <1\}=\{s:\|s-\nabla B(x)\|'_{\nabla B(x)} <1\}\subset\rmint (-\cK^*).
\]
Taking the closure on both sides of this relation implies that statement (vi) holds.
\end{proof}

The following lemma shows that $[\nabla^2 B(x)]^{-1}$ is bounded in the intersection of a unit sphere and $\rmint \cK$, which is crucial for the development of this paper.
	
\begin{lemma} \label{lem:gamma}
The matrix $[\nabla^2 B(x)]^{-1}$ is bounded in the intersection of a unit sphere and $\rmint \cK$, that is,  $\gamma<\infty$, 
where
\beq \label{gamma}
\gamma := \underset{x\in\rmint \cK,\|x\|=1}{\sup}\|[\nabla^2 B(x)]^{-1}\|.
\eeq
\end{lemma}
	
\begin{proof}
Let $x\in\rmint \cK$ with $\|x\|=1$ be arbitrarily chosen, $y$ a fixed interior point of $\cK$, and $\tr=\|x-y\|$. Then there exists some $r>0$ such that the  Euclidean ball centered at $y$ with radius $r$ is included in $\cK$. By this and the convexity of $\cK$, one can observe that $x+\alpha(y-x) \in \cK$ for all $\alpha\in[0, (\tr+r)/\tr]$.\footnote{By convention, $\delta/0$ is set to $\infty$ for any $\delta>0$ throughout this paper.} It then follows that 
$\pi_x(y) \le \tr/(\tr+r)$, where $\pi_x(\cdot)$ is  the Minkowski function of $\cK$ with the pole at $x$ defined as
\[
\pi_x(z) = \inf\{t>0: x+t^{-1}(z-x) \in \cK\} \quad \forall z.
\]
Notice that $\|x\|=1$ and $\tr=\|x-y\|$. Hence, $\tr\le 1+\|y\|$, which together with $\pi_x(y) \le \tr/(\tr+r)$ implies that 
\[
\pi_x(y) \le \frac{1+\|y\|}{1+\|y\|+r} =: \Delta_y.
\]
By this and \cite[Equation (3.16)]{N04IPM}, one has 
\vspace{-1mm}
\[
\nabla^2 B(y) \preceq \left(\frac{\vartheta+2\sqrt{\vartheta}}{1-\pi_x(y)}\right)^2\nabla^2 B(x) 
\preceq \left(\frac{\vartheta+2\sqrt{\vartheta}}{1-\Delta_y}\right)^2\nabla^2 B(x).
\]
It follows that 
\[
0 \preceq [\nabla^2 B(x)]^{-1} \preceq \left(\frac{\vartheta+2\sqrt{\vartheta}}{1-\Delta_y}\right)^2 [\nabla^2 B(y)]^{-1}.
\]
Using this,  $\Delta_y\in(0,1)$ and the arbitrary choice of $x$, we conclude that 
\vspace{-1mm}
\[
\underset{x\in\rmint \cK,\|x\|=1}{\sup}\|[\nabla^2 B(x)]^{-1}\| \le \left(\frac{\vartheta+2\sqrt{\vartheta}}{1-\Delta_y}\right)^2 \|[\nabla^2 B(y)]^{-1}\| < \infty.
\]
\end{proof}

The following theorem shows that $\|[\nabla^2 B(x)]^{-1}\|$ is at most in the order of $\|x\|^2$ 
for all $x\in\rmint\cK$.

\begin{theorem} \label{bound-nabla-2B}
Let $\gamma$ be defined in \eqref{gamma}. Then
$\|[\nabla^2 B(x)]^{-1}\|\le \gamma\|x\|^2$  for every $x\in\rmint\cK$, and $[\nabla^2 B(x)]^{-1}$ is bounded in any nonempty bounded subset of $\rmint \cK$.
\end{theorem}

\begin{proof}
Differentiating both sides of \eqref{def:log-hm-self-concordant-barrier} twice with respect to $x$, we have
\[
t^2 \nabla^2 B(tx) = \nabla^2 B(x)\  \quad\forall x\in\rmint \cK, \ t>0.
\]
Letting $t=1/\|x\|$, we further obtain that
\[
\frac{1}{\|x\|^2}\nabla^2 B\left(\frac{x}{\|x\|}\right)=\nabla^2 B(x)\ \quad\forall x\in\rmint \cK.
\]
It then follows that  
\[
[\nabla^2 B(x)]^{-1} = \|x\|^2 [\nabla^2 B(x/\|x\|)]^{-1}\ \quad\forall x\in\rmint \cK,
\]
which together with \eqref{gamma} implies that $\|[\nabla^2 B(x)]^{-1}\|\le \gamma\|x\|^2$ for every $x\in\rmint \cK$. It immediately follows that $[\nabla^2 B(x)]^{-1}$ is bounded in any nonempty bounded subset of $\rmint \cK$.
\end{proof}

Note that $[\nabla^2 B(x)]^{-1}$ is well-defined in $\rmint \cK$ but undefined on the boundary of $\cK$. To capture its behavior as $x$ approaches the boundary of $\cK$, we next introduce a terminology called the {\it limiting inverse of the Hessian of B}, denoted by $\nabla^{-2} B$, which is a generalization of $[\nabla^2 B]^{-1}$.

\begin{definition}[{{\bf limiting inverse of the Hessian of $B$}}]\label{inverse-B2}
\vspace{-1mm}
\beq \label{nablaB-2}
\nabla^{-2} B (x)=\left\{M: M=\lim\limits_{k\to\infty}[\nabla^2B(x^k)]^{-1}\  \mbox{for some} \  
\{x^k\}\subset\rmint\cK \ \mbox{with} \ x^k\to x \ \mbox{as} \ k\to\infty \right\}\ \quad  \forall x\in\cK. 
\vspace{-2.5mm}
\eeq
\end{definition}
From Theorem \ref{bound-nabla-2B}, we know that $[\nabla^2 B(x)]^{-1}$ is bounded in any nonempty bounded subset of $\rmint \cK$, which implies that $\nabla^{-2} B (x) \neq \emptyset$ for every $x\in\cK$. In addition, since $[\nabla^2 B(x)]^{-1}$ is continuous in $\rmint \cK$, one can see that $\nabla^{-2} B (x)$ becomes a singleton $\{[\nabla^2 B(x)]^{-1}\}$ for any $x\in \rmint \cK$. Thus, $\nabla^{-2} B$ is indeed a generalization of $[\nabla^2 B]^{-1}$.

Notice that Lemma \ref{lem:barrier-property}(iii) only holds at any $x\in\rmint\cK$. With the aid of the limiting inverse of $\nabla^2 B$, we next generalize Lemma \ref{lem:barrier-property}(iii) to the one that holds at every point in $\cK$.
\vspace{-1mm}
\begin{theorem} \label{boundary-ball} 
For any $x\in\cK$, it holds that
\vspace{-1mm}
\[
\{x+M^{1/2}d : \|d\|<1\}\subseteq \cK\ \quad \forall  M\in \nabla^{-2} B (x).
\vspace{-1mm}
\]
\end{theorem}
\begin{proof}
Let  $M\in \nabla^{-2} B (x)$ be arbitrarily chosen. It then follows from \eqref{nablaB-2} that there exists some $\{x^k\}\subset\rmint\cK$ such that $x^k\to x$ 
and $[\nabla^2B(x^k)]^{-1} \to M$ as $k\to\infty$. By the nonsingularity of $\nabla^2B(x^k)$, the definition of $\|\cdot\|_{x^k}$, and Lemma~\ref{lem:barrier-property}(iii), one can observe that
\vspace{-1mm}
\[
\{x^k+[\nabla^2B(x^k)]^{-1/2}d: \|d\|<1\}=\{y:\|y-x^k\|_{x^k}<1\}\subset \rmint\cK.
\vspace{-1mm}
\]
Taking limit on both sides of this relation as $k\to\infty$, we obtain that $\{x+M^{1/2}d : \|d\|<1\}\subseteq \cK$. Hence, the conclusion holds.
\vspace{-0.5mm}
\end{proof}

\section{Optimality conditions}\label{sec:approx-opt-cond}
In this section we study optimality conditions of problem \eqref{conic-prob}. In particular, we first derive some first- and second-order optimality conditions for \eqref{conic-prob}, and then introduce a definition of approximate first- and second-order stationary points of \eqref{conic-prob}.

Suppose that $x^*$ is a local minimizer of problem \eqref{conic-prob}. By this and the assumption that Slater's condition holds for \eqref{conic-prob}, it follows that there exists a Lagrangian multiplier $\lambda^*\in\bR^m$ such that
\vspace{-1mm}
\beq
\nabla f(x^*) + A^T\lambda^*\in-\cN_{\cK}(x^*). \label{stationary} 
\vspace{-1mm}
\eeq
This is a classical first-order optimality condition of problem  \eqref{conic-prob}. One can easily obtain an inexact counterpart of it. However, its inexact counterpart is not 
suitable for the design and analysis of a Newton-CG based barrier method for solving \eqref{conic-prob}. Due to this, we next derive an alternative first-order optimality condition for \eqref{conic-prob}. 

\begin{theorem}[{{\bf first-order optimality condition}}]\label{thm:1st-opt-cond}
Let $x^*$ be a local minimizer of problem \eqref{conic-prob} and $M\in \nabla^{-2} B (x^*)$ be arbitrarily chosen. Suppose that $f$ is continuously differentiable at $x^*$. Then there exists a Lagrangian multiplier $\lambda^*\in\bR^m$ such that 
\vspace{-1.5mm}
\begin{eqnarray}
&&\nabla f(x^*)+A^T\lambda^*\in\cK^*,\label{1st-order-1}\\
&&M^{1/2}(\nabla f(x^*)+A^T\lambda^*)=0. \label{1st-order-2}
\vspace{-2mm}
\end{eqnarray}
\end{theorem}
\begin{proof}
Since $x^*$ is a local minimizer of \eqref{conic-prob}, we know from above that there exists a Lagrangian multiplier $\lambda^*\in\bR^m$ such that \eqref{stationary} holds.
Note that $\cK$ is a closed convex cone. It is not hard to verify $-\cN_{\cK}(x^*)\subseteq\cK^*$, which along with \eqref{stationary} leads to  \eqref{1st-order-1}.

We next prove \eqref{1st-order-2}. Since $M\in \nabla^{-2} B (x^*)$, it follows from Theorem \ref{boundary-ball} that $\{x^*+M^{1/2}d : \|d\|<1\}\subseteq \cK$. By this and  \eqref{stationary}, one has
\vspace{-1.5mm} 
\[
d^TM^{1/2}(\nabla f(x^*)+A^T\lambda^*)\ge 0 \ \quad \forall d \ \mbox{with} \ \|d\|<1,
\vspace{-0.5mm}
\]
which implies $M^{1/2}(\nabla f(x^*)+A^T\lambda^*)=0$, and hence  \eqref{1st-order-2} holds as desired.
\end{proof}
	
The first-order optimality conditions \eqref{1st-order-1} and \eqref{1st-order-2} appear to be different from the classical one \eqref{stationary}. Nonetheless, the following proposition shows that they are essentially equivalent, and both are related to the complementary slackness condition \eqref{slack-cond}.
	
\begin{proposition}
Let $x^*\in\cK$, $\lambda^*\in\bR^m$, and $M\in \nabla^{-2} B (x^*)$ be given. Then the following statements hold.
\begin{enumerate}[{\rm (i)}]
\item  The relations \eqref{1st-order-1} and \eqref{1st-order-2} hold if and only if  \eqref{stationary} holds.
\item The relation \eqref{1st-order-1} and the complementary slackness condition 
\beq
\langle x^*,\nabla f(x^*)+A^T\lambda^*\rangle=0 \label{slack-cond}
\eeq
 hold if and only if \eqref{stationary} holds.
\item The relations \eqref{1st-order-1} and \eqref{1st-order-2} hold if and only if \eqref{1st-order-1} and \eqref{slack-cond} hold.
\end{enumerate}
\end{proposition}
	
\begin{proof}
Firstly, by the same argument as used in the proof of Theorem \ref{thm:1st-opt-cond}, one can see that if \eqref{stationary} holds, then \eqref{1st-order-1} and \eqref{1st-order-2} hold. 

Secondly, we show that if \eqref{1st-order-1} and \eqref{1st-order-2} hold, then \eqref{slack-cond} holds. To this end, suppose that \eqref{1st-order-1} and \eqref{1st-order-2} hold. Since $M\in \nabla^{-2} B (x^*)$, it follows from \eqref{nablaB-2} that there exists some $\{x^k\}\subset\rmint\cK$ such that $x^k\to x^*$ and $[\nabla^2B(x^k)]^{-1} \to M$ as $k\to\infty$. By these, \eqref{1st-order-2}, and Lemma \ref{lem:barrier-property}(i),  one has that
\begin{equation}\nonumber
\begin{array}{rcl}
|\langle x^*,\nabla f(x^*)+A^T\lambda^*\rangle|&=&\lim\limits_{k\to\infty}|\langle x^k,\nabla f(x^k)+A^T\lambda^*\rangle|\\[8pt]
&\le&\lim\limits_{k\to\infty}\|[\nabla^2 B(x^k)]^{1/2}x^k\|\|[\nabla^2B(x^k)]^{-1/2}(\nabla f(x^k)+A^T\lambda^*)\|  \\[8pt]
&=&\sqrt{\vartheta}\|M^{1/2}(\nabla f(x^*)+A^T\lambda^*)\|=0,
\end{array}
\end{equation}
where the inequality uses Cauchy-Schwarz inequality. Hence, \eqref{slack-cond} holds as desired.

Thirdly, we show that if  \eqref{1st-order-1} and \eqref{slack-cond} hold, then \eqref{stationary} holds. To this end, suppose that  \eqref{1st-order-1} and \eqref{slack-cond} hold. 
Then we have 
\[
\langle x-x^*,\nabla f(x^*)+A^T\lambda^*\rangle\overset{\eqref{slack-cond}}{=}\langle x,\nabla f(x^*)+A^T\lambda^*\rangle\overset{\eqref{1st-order-1}}{\ge} 0\ \quad 
\forall x\in\cK,
\]
which yields $\nabla f(x^*)+A^T\lambda^*\in -\cN_{\cK}(x^*)$, and hence \eqref{stationary} holds.

Combining the above arguments, we can conclude that statements (i), (ii) and (iii) hold.\end{proof}

The classical second-order optimality condition for constrained optimization problems was well studied in the literature (e.g., see \cite{NW06}).  It can be easily specialized to problem \eqref{conic-prob}. However, 
its verification is generally hard since a sophisticated critical cone is involved (e.g., see \cite{MK87,PS88}). 
We next derive a weaker yet verifiable second-order optimality condition. Strictly speaking, it shall be called a {\it weak second-order optimality condition}. 
For the ease of reference, we simply call it a second-order optimality condition.

\begin{theorem}[{{\bf second-order optimality condition}}]\label{thm:2nd-opt-cond}
Let $x^*$ be a local minimizer of problem \eqref{conic-prob} and $M\in \nabla^{-2} B (x^*)$ be arbitrarily chosen. Suppose that $f$ is twice continuously differentiable at $x^*$. Then there exists a Lagrangian multiplier $\lambda^*\in\bR^m$ such that \eqref{1st-order-1}, \eqref{1st-order-2}, and additionally  
\beq
d^TM^{1/2}\nabla^2 f(x^*) M^{1/2} d\ge0\ \quad \forall d\in\cC(M) \label{2nd-order}
\eeq
hold, where 	
\beq \label{CM}
\cC(M) :=\{d:AM^{1/2}d=0\}.
\eeq	
\end{theorem}

\begin{proof}
It follows from Theorem \ref{thm:1st-opt-cond} that  \eqref{1st-order-1} and \eqref{1st-order-2} hold. We now prove \eqref{2nd-order}. Indeed, it suffices to prove that \eqref{2nd-order} holds for any $d\in\cC(M)$ with $\|d\| \le 1$. To this end, let $d\in\cC(M)$ with $\|d\|\le1$ be arbitrarily chosen. By this, $M\in \nabla^{-2} B (x^*)$ and Theorem \ref{boundary-ball}, one has that $\{x^*+tM^{1/2}d : t \in (-1,1)\}\subseteq \cK$ and $A(x^*+tM^{1/2}d)=b$. In view of these and the fact that $x^*$ is a local minimizer of \eqref{conic-prob},  we can observe that $t^*=0$ is a local minimizer of the problem
\[
\min\limits_{t\in (-1,1)} \left\{\psi(t)=f(x^*+tM^{1/2}d)\right \}.
\]
By its second-order necessary optimality condition at $t^*=0$, one has that
\[
0 \le \psi^{\prime\prime}(0)=d^TM^{1/2}\nabla^2 f(x^*)M^{1/2}d
\]
for any $d\in\cC(M)$ with $\|d\|\le1$.  It implies that  the relation \eqref{2nd-order} holds.
\end{proof}

Theorems \ref{thm:1st-opt-cond} and \ref{thm:2nd-opt-cond} provide first- and second-order necessary optimality conditions for problem \eqref{conic-prob}. For convenience, we refer to a feasible point $x^*$ of \eqref{conic-prob} as {\it a first-order stationary point} of \eqref{conic-prob} if it together with some $\lambda^*\in \bR^m$ satisfies \eqref{1st-order-1} and \eqref{1st-order-2}. We further refer to it as {\it a second-order stationary point} of \eqref{conic-prob} if it additionally satisfies \eqref{2nd-order}. Due to the sophistication of the problem, it is generally impossible to find an exact  first- or second-order stationary point of \eqref{conic-prob}.  Instead, we are interested in finding an approximate counterpart of them that is defined as follows. 

\begin{definition}[{{\bf $\epsilon_g$-first-order stationary point}}] 
\label{approx-1st-stationary}
For any $\epsilon_g>0$, a point $x$ is called an $\epsilon_g$-first-order stationary point ($\epsilon_g$-FOSP) of \eqref{conic-prob} if it together with some $\lambda\in\bR^m$ satisfies
\begin{eqnarray}
&&Ax=b,\ x\in\rmint\cK,\label{inexact-feasible}\\
&&\nabla f(x)+A^T\lambda\in\cK^*,\label{inexact-1st-order-1}\\
&&\|\nabla f(x)+A^T\lambda\|_x^*\le\epsilon_g.\label{inexact-1st-order-2}
\end{eqnarray}
\end{definition}	

\begin{definition}[{{\bf $(\epsilon_g,\epsilon_H)$-second-order stationary point}}]
\label{approx-2nd-stationary}
For any $\epsilon_g, \epsilon_H>0$, a point $x$ is called an $(\epsilon_g,\epsilon_H)$-second-order stationary point ($(\epsilon_g,\epsilon_H)$-SOSP) of \eqref{conic-prob} if it together with some $\lambda\in\bR^m$ satisfies \eqref{inexact-feasible}-\eqref{inexact-1st-order-2} and additionally 
\beq  \label{inexact-2nd-order-old}
d^T[\nabla^2 B(x)]^{-1/2}\nabla^2 f(x) [\nabla^2 B(x)]^{-1/2} d\ge -\epsilon_H\|d\|^2\ \quad \forall d\in\cC([\nabla^2 B(x)]^{-1}),
\eeq
where $\cC(\cdot)$ is defined in \eqref{CM}.
\end{definition}

\begin{remark}
\begin{enumerate}[{\rm (i)}]
\item One can see that if a point $x\in\rmint\cK$ satisfies  \eqref{inexact-1st-order-2} and 
\eqref{inexact-2nd-order-old}, then it nearly satisfies \eqref{1st-order-2} and \eqref{2nd-order} with $x^*$ replaced by $x$. Thus, the $\epsilon_g$-FOSP and $(\epsilon_g,\epsilon_H)$-SOSP introduced in Definitions \ref{approx-1st-stationary} and \ref{approx-2nd-stationary} are indeed an approximate counterpart of the FOSP and SOSP of problem \eqref{conic-prob}. In addition, when $\cK=\bR^n_+$, they are stronger than the approximate FOSP and SOSP introduced in \cite{HLY19,OW21,XW21pN} for problem \eqref{boxconstr-prob} or \eqref{lpconstr-prob}. Also, for a general cone $\cK$, they are stronger than the ones introduced in \cite{DS21HBA}. Specifically, the approximate FOSP and SOSP found by the methods in \cite{DS21HBA,HLY19,OW21,XW21pN} satisfy \eqref{inexact-feasible}, \eqref{inexact-1st-order-2} and \eqref{inexact-2nd-order-old} respectively, while only approximately satisfying \eqref{inexact-1st-order-1}.

\item Upon a suitable change of variable, one can see that \eqref{inexact-2nd-order-old} is equivalent to
\begin{equation}\label{inexact-2nd-order}
d^T\nabla^2 f(x)d\ge-\epsilon_H\|d\|_x^2\ \quad\forall d\in \{d:Ad=0\}.
\end{equation}
\item The relations \eqref{inexact-1st-order-2} and \eqref{inexact-2nd-order} involve the local norms $\|\cdot\|^*_x$ and $\|\cdot\|_x$.  It is interesting to observe that they possess a scale-invariant property. That is, they hold at a point $x\in\rmint \cK$ for problem \eqref{conic-prob} if and only if they hold at a point $y=W^{-1}x\in \rmint(W^{-1}\mathcal{K})$ for the problem
\beq \label{cone-prob2}
\min_y\{f(Wy):AWy=b,\ y\in W^{-1}\mathcal{K}\},
\eeq
where $W$ is a nonsingular matrix. It shall be noted that $B(Wy)$ is an LHSC barrier function for the cone $W^{-1}\mathcal{K}$ and the local norms used in \eqref{inexact-1st-order-2} and \eqref{inexact-2nd-order} for problem \eqref{cone-prob2} are defined in terms of the barrier function $B(Wy)$. 
\end{enumerate}
\end{remark}

\section{A Newton-CG based barrier method}\label{sec:algorithm}

In this section we develop a Newton-CG based barrier (NCGB) method for finding an approximate  second-order stationary point of problem \eqref{conic-prob}. Instead of solving \eqref{conic-prob} directly, the {NCGB} method solves by a preconditioned Newton-CG method the barrier problem 
\begin{equation}\label{barrier-prob}
\min_x\left\{\phi_{\mu}(x):=f(x)+\mu B(x)\right\} \quad\st\quad Ax=b
\end{equation}
for a suitable choice of parameter $\mu>0$.  In particular,  we first introduce a damped preconditioned Newton system and review a capped CG method for solving it in Subsections \ref{sec:precond-matrix} and \ref{sec:capped-cg-meo}, respectively. Then we present a minimum eigenvalue oracle in Subsection \ref{eigval-orac} that can be used to estimate the minimum eigenvalue of a real symmetric matrix. Finally, we present a NCGB method for solving  problem \eqref{conic-prob} in Subsection \ref{BNCG}.

\subsection{Damped preconditioned Newton system}\label{sec:precond-matrix}

In this subsection we introduce a damped preconditioned Newton system that will be used subsequently to develop a NCGB method for solving problem \eqref{conic-prob}.  

Since our goal is to find an approximate second-order stationary point of problem \eqref{conic-prob}, it would be  natural to apply the classical projected Newton method to solve \eqref{barrier-prob}. However, ill-conditioning could be an issue for this method. To see this,  suppose that $x^k$ is a current approximate solution to \eqref{barrier-prob} that satisfies $Ax^k=b$ and $x^k \in \rmint \cK$. To generate the next iterate $x^{k+1}$, the classical projected Newton method attempts to find a search direction by solving the subproblem
\beq \label{Newton-prob1}
\min_\bd\ \nabla \phi_{\mu}(x^k)^T\bd + \frac12 \bd^T \nabla^2 \phi_{\mu}(x^k)\bd \quad\st\quad A\bd=0.
\eeq
Notice that $\nabla^2 \phi_{\mu}(x^k)$ becomes ill-conditioned as $x^k$ is close to the boundary of $\cK$, which could cause iterative methods to converge slowly when applied to solve \eqref{Newton-prob1}. To remedy this, we instead consider the following preconditioned subproblem
\beq \label{Newton-prob2}
\min_{\td}\ \nabla \phi_{\mu}(x^k)^TM_k \td + \frac12 \td^T M_k^T\nabla^2 \phi_\mu(x^k)M_k \td \quad\st\quad AM_k\td=0,
\eeq
which is obtained from \eqref{Newton-prob1} by letting $\bd=M_k\td$, 
where $M_k$ is a matrix such that 
\begin{equation}\label{Mk}
[\nabla^2 B(x^k)]^{-1}=M_kM_k^T.\footnote{As will be discussed  in Section \ref{sec:complexity}, there is no need to compute such $M_k$ explicitly. }
\end{equation}
 
Let $Q_k$ denote the projection matrix for the projection from $\bR^n$ to the null space of $AM_k$, that is,  
\begin{equation}\label{Qk}
Q_k=I-M_k^TA^T(AM_kM_k^TA^T)^{-1}AM_k.
\end{equation}
By letting $\td=Q_k\hd$, one can see that \eqref{Newton-prob2} is equivalent to 
\[
\min_{\hd}\ \nabla \phi_{\mu}(x^k)^T M_kQ_k \hd + \frac12 \hd^T Q_k^TM_k^T\nabla^2\phi_\mu(x^k)M_kQ_k\hd,
\]
which leads to a preconditioned (projected) Newton system 
\beq \label{Newton-prob3}
(P_k^T\nabla^2\phi_\mu(x^k)P_k)\hd=-P_k^T\nabla\phi_\mu(x^k),
\eeq
where
\begin{equation}\label{Pk}
P_k=M_kQ_k=M_k-M_kM_k^TA^T(AM_kM_k^TA^T)^{-1}AM_k.
\end{equation}
For a similar reason as pointed out in \cite{ROW20} for smooth nonconvex unconstrained optimization, CG method, when applied to \eqref{Newton-prob3}, may not be able to produce a sufficient descent direction for \eqref{barrier-prob}. Therefore, we instead consider a damped counterpart of \eqref{Newton-prob3}, namely, the {\it damped preconditioned Newton system}
\begin{equation}\label{DN-system-2}
(P_k^T\nabla^2\phi_\mu(x^k)P_k+ 2\sqrt{\epsilon} I)\hd=-P_k^T\nabla\phi_\mu(x^k)
\end{equation}
for some $\epsilon>0$. 
In the next subsection, we review a capped CG method proposed in \cite{ROW20} that can be suitably applied to \eqref{DN-system-2} for finding a sufficient descent direction for \eqref{barrier-prob}.

\subsection{A capped conjugate gradient method}
\label{sec:capped-cg-meo}

In this subsection we review a capped conjugate gradient (CG) method that was proposed in \cite{ROW20} for solving a possibly indefinite linear system
\beq  \label{linsys}
(H+2\varepsilon I)\hd=-g,
\eeq 
where $0 \neq g\in\bR^n$, $\varepsilon>0$, and $H\in\bR^{n\times n}$ is a symmetric matrix.  This capped CG method  is a modification of the classical CG method (e.g., see \cite{NW06}). It terminates within a finite number of iterations, and outputs either an approximate solution $\hd$ of \eqref{linsys} satisfying $\|(H+2\varepsilon I)\hd+g\| \le \widehat \zeta \|g\|$ and $\hd^T H \hd \ge -\varepsilon \|\hd\|^2$ for some $ \widehat \zeta \in (0,1)$ or a direction $\hd$ such that $\hd^TH \hd < -\varepsilon \|\hd\|^2$. For the ease of latter reference, these two types of outputs are classified by SOL and NC, respectively.\footnote{SOL and NC stand for `approximate solution' and `negative curvature', respectively.} The capped CG method  \cite{ROW20} is presented in Algorithm \ref{alg:capped-CG} in Appendix \ref{appendix:capped-CG}. 
Its detailed motivation and explanation can be found in  \cite{ROW20}. 
This method will be subsequently applied to the damped preconditioned Newton system \eqref{DN-system-2} arising in NCGB method for finding a sufficient descent direction for \eqref{barrier-prob}.

The following theorem states some properties of Algorithm \ref{alg:capped-CG}. 
	
\begin{theorem}\label{lem:capped-CG-cmplxity}
Consider applying Algorithm \ref{alg:capped-CG} to the linear system \eqref{linsys} with 
$g\neq 0$, $\varepsilon>0$, and $H$ being a ${n\times n}$ symmetric matrix. Then the following statements hold.
\begin{enumerate}[{\rm (i)}]
\item The output $\hd$ of Algorithm \ref{alg:capped-CG} is a nonzero vector.
\item The number of iterations of Algorithm \ref{alg:capped-CG} is $\widetilde{\cO}(\min\{n,\varepsilon^{-1/2}\})$.
\end{enumerate}
\end{theorem}

\begin{proof}
 (i) One can observe that the output $d$ of Algorithm \ref{alg:capped-CG} satisfies 
$\|(H+2\varepsilon I)\hd+g\| \le \widehat \zeta \|g\|$
or $\hd^TH \hd < -\varepsilon \|\hd\|^2$. By this, $g \neq 0$ and $ \widehat \zeta \in (0,1)$, one can easily see that $\hd \neq 0$.

(ii) From \cite[Lemma~1]{ROW20}, we know that the number of iterations of Algorithm \ref{alg:capped-CG} is bounded by $\min\{n,J(U,\varepsilon,\zeta)\}$, where $J(U,\varepsilon,\zeta)$ is the smallest integer $J$ such that $\sqrt{T}\tau^{J/2}\le\widehat{\zeta}$, where $U,\widehat{\zeta},T$ and $\tau$ are the values returned by Algorithm \ref{alg:capped-CG}. In addition, it was shown in \cite[Section 3.1]{ROW20} that 
\[
J(U,\varepsilon,\zeta)\le\left\lceil \left(\sqrt{\kappa}+\frac{1}{2}\right)\ln\left(\frac{144(\sqrt{\kappa}+1)^2\kappa^6}{\zeta^2}\right)\right\rceil,
\]
where $\kappa={\cO}(\varepsilon^{-1})$ is an output by Algorithm \ref{alg:capped-CG}. 
Then one can see that $J(U,\varepsilon,\zeta)=\widetilde{\cO}(\varepsilon^{-1/2})$. It thus follows that the number of iterations of Algorithm \ref{alg:capped-CG} is $\widetilde{\cO}(\min\{n,\varepsilon^{-1/2}\})$.
\end{proof}
	
\subsection{A minimum eigenvalue oracle}
\label{eigval-orac}	

In this subsection we present a minimum eigenvalue oracle (Algorithm \ref{pro:meo}), which will subsequently be used to check whether the second-order optimality condition of problem \eqref{conic-prob} nearly holds at a given point.  In particular, given a symmetric matrix $H$ and $\varepsilon>0$, this oracle either certifies $\lambda_{\min}(H)\ge-\varepsilon$ with high probability  or finds a unit vector $v$ such that $v^THv \le -\varepsilon/2$. The Lanczos method is often used as a solver in this oracle (e.g., see \cite{CDHS17,OW21,ROW20}).  
 
\begin{algorithm}
\caption{A minimum eigenvalue oracle}
\label{pro:meo}
{\small
\noindent\textit{Input}: symmetric matrix $H\in\bR^{n\times n}$, tolerance $\varepsilon>0$, and probability parameter $\delta\in(0,1)$.\\
\noindent\textit{Output:} a sufficiently negative curvature direction $v$ satisfying $v^THv\le-\varepsilon/2$ and $\|v\|=1$; or a certificate that $\lambda_{\min}(H)\ge-\varepsilon$ with probability at least  $1-\sqrt{2.75n}\delta^{\|H\|^{-1/2}}$.\\
 Apply the Lanczos method \cite{KW92LR} to estimate $\lambda_{\min}(H)$ starting with a random vector uniformly generated on the unit sphere, and run it for at most 
\beq \label{N-iter}
N(\varepsilon,\delta) := \min\left\{n,1+\left\lceil \varepsilon^{-1/2}\ln\delta^{-1}\right\rceil\right\}
\eeq
iterations. 
\bi
\item[(i)]
If a unit vector $v$ with $v^THv \le -\varepsilon/2$ is found at some iteration, terminate  and return $v$.
\item[(ii)]Otherwise, it certifies that $\lambda_{\min}(H)\ge-\varepsilon$ holds with probability at least $1-\sqrt{2.75n}\delta^{\|H\|^{-1/2}}$.
\ei
}
\end{algorithm}

The following theorem justifies that Algorithm \ref{pro:meo} can produce a desirable output after running the Lanczos method for a certain number of iterations. Its proof directly follows from \cite[Lemma~2]{ROW20}.

\begin{theorem} \label{rand-Lanczos}
Consider Algorithm~\ref{pro:meo} with tolerance $\varepsilon>0$, probability parameter $\delta\in(0,1)$, and symmetric matrix $H\in\bR^{n\times n}$ as its input. Let $N(\varepsilon,\delta)$ be defined in \eqref{N-iter}. Then Algorithm~\ref{pro:meo} runs at most $N(\varepsilon,\delta)$ iterations. Moreover, it either finds a sufficiently negative curvature direction $v$ satisfying $v^THv\le-\varepsilon/2$ and $\|v\|=1$; or provides a certificate that $\lambda_{\min}(H)\ge-\varepsilon$ holds with probability at least  $1-\sqrt{2.75n}\delta^{\|H\|^{-1/2}}$.
\end{theorem}
	
\begin{remark}
Generally, computing $\|H\|$ may not be cheap when $n$ is large. 
Nevertheless, $\|H\|$ can be efficiently estimated by a randomization scheme with high confidence  (e.g., see the discussion in \cite[Appendix~B3]{ROW20}).
\end{remark}

\subsection{A Newton-CG based barrier method for problem \eqref{conic-prob}} \label{BNCG}

In this subsection we propose a Newton-CG based barrier (NCGB) method for solving problem \eqref{conic-prob}. In each iteration, our NCGB method starts by checking whether the current iterate $x^k$ and the associated Lagrangian multiplier estimates $\lambda_{k}^{(1)}$ and $\lambda_{k}^{(2)}$ satisfy certain approximate first-order optimality conditions of \eqref{conic-prob}.  If not, then the capped CG method (Algorithm \ref{alg:capped-CG}) is applied to the damped preconditioned Newton system \eqref{DN-system-2} to obtain either an inexact damped Newton direction or a sufficiently negative curvature direction, and the next iterate $x^{k+1}$ is generated by performing a line search along this direction. Otherwise, the current iterate $x^k$ is already an approximate first-order stationary point of \eqref{conic-prob}, and a minimum eigenvalue oracle (Algorithm \ref{pro:meo}) is further invoked to either obtain a sufficiently negative curvature direction and generate the next iterate $x^{k+1}$ via a line search, or certify that $x^k$ is an approximate SOSP of \eqref{conic-prob} with high probability and terminate the method. 

For the convenience of presentation, we let   
\begin{equation}\label{Rk}
R_k = -(AM_kM_k^TA^T)^{-1}AM_kM_k^T,
\end{equation}
where $M_k$ satisfies \eqref{Mk}. In view of \eqref{Pk} and \eqref{Rk}, it is easy to verify that
\begin{equation}\label{eq:relation-Pk-Rk}
P_k =(I+R_k^TA)M_k.
\end{equation}

We are now ready to present our NCGB method in Algorithm  \ref{alg:BNCG} for solving problem \eqref{conic-prob}, in which $Q_k$, $P_k$ and $R_k$ are defined in \eqref{Qk}, \eqref{Pk} and \eqref{Rk}, respectively. The study of its complexity results is deferred to Section \ref{sec:complexity}. In what follows, we make some remarks about Algorithm  \ref{alg:BNCG}.

\begin{algorithm}
{\small
\caption{A Newton-CG based barrier method for \eqref{conic-prob}}
\label{alg:BNCG}
\begin{algorithmic}
\State Let $P_k$, $Q_k$ and $R_k$ be defined in \eqref{Pk}, \eqref{Qk} and \eqref{Rk}, respectively.
\State \noindent\textit{Input}: $\epsilon\in(0,1)$,  $x^0\in\Omega^{\rm o}$,  $\zeta\in(0,1)$, $\beta\in[\sqrt{\epsilon},1)$,  $\theta\in(0,1)$,  $\eta\in(0,1)$, $\delta\in(0,1)$, and  $\vartheta\ge1$ (the parameter of $B$).
\State Set 
\[
x^{-1}=x^0,\quad d^{-1}=0,\quad {\mu=\frac{(1-\beta)\epsilon}{2((1-\beta)^2+\sqrt{\vartheta})}}, \quad \mbox{d$\_$type=NC}, \quad \alpha_{-1}=0, \quad \lambda^{(2)}_{-1}=0;
\]
\For{$k=0,1,2,\ldots$}
\State Set $\lambda^{(1)}_k\leftarrow R_k\nabla\phi_{\mu}(x^k)$, where $R_k$ is given in \eqref{Rk};
 \If{d$\_$type=SOL and $\alpha_{k-1}=1$}
\State 
$\lambda^{(2)}_k \leftarrow R_{k-1}(\nabla^2 f(x^{k-1})P_{k-1}d^{k-1}+\nabla\phi_{\mu}(x^{k-1}))$;
\Else
\State  $\lambda^{(2)}_k \leftarrow \lambda^{(2)}_{k-1}$;
\EndIf
\If{$\min\{\|\nabla f(x^k)+A^T\lambda_k^{(1)}+\mu\nabla B(x^k)\|_{x^k}^*,\|\nabla f(x^k)+A^T\lambda_k^{(2)}+\mu\nabla B(x^{k-1})\|_{x^k}^*\}>(1-\beta)\mu$}
\State Call Algorithm \ref{alg:capped-CG} with $H=P_k^T\nabla^2\phi_\mu(x^k)P_k,\ \varepsilon=\sqrt{\epsilon},\ g=P_k^T\nabla\phi_\mu(x^k)$, accuracy parameter $\zeta$, and 
\State bound $U=0$ to obtain  outputs $\widehat{d}^k$, d$\_$type;
\If{d$\_$type=NC}
\beq \label{dk-nc}
d^k\leftarrow -\sgn(g^T\widehat{d}^k)\min\left\{\frac{|(\widehat{d}^k)^TP_k^T\nabla^2\phi_\mu(x^k)P_k\widehat{d}^k|}{\|\widehat{d}^k\|^3},\frac{\beta}{\|Q_k\widehat{d}^k\|}\right\}\widehat{d}^k;
\eeq
\Else\ \{d$\_$type=SOL\}
\beq \label{dk-sol}
d^k\leftarrow\min\left\{1,\frac{\beta}{\|Q_k\widehat{d}^k\|}\right\}\widehat{d}^k;
\eeq
\EndIf
\State Go to {\bf Line Search};
\Else
\State Call Algorithm \ref{pro:meo} with $H=P_k^T\nabla^2 f(x^k)P_k$, $\varepsilon=\sqrt{\epsilon}$, and $\delta>0$; 
\If{Algorithm \ref{pro:meo} certifies that $\lambda_{\min}(P_k^T\nabla^2 f(x^k)P_k)\ge-\sqrt{\epsilon}$}
\State Output $x^k$ and terminate;
\Else\ $\{$Sufficiently negative curvature direction $v$ returned by Algorithm \ref{pro:meo}$\}$
\State Set 
\beq \label{dk-2nd-nc}
d^k\leftarrow -\sgn(v^T P_k^T\nabla\phi_\mu(x^k))\min\left\{|v^TP_k^T\nabla^2\phi_\mu(x^k)P_kv|,{\frac{\beta}{\|Q_kv\|}}\right\}v;
\eeq
\State Go to {\bf Line Search};
\EndIf
\EndIf
\State{\bf Line Search:} 
\If{d$\_$type=SOL}
\State Find $\alpha_k=\theta^{j_k}$, where $j_k$ is the smallest nonnegative integer $j$ such that
\begin{equation}\label{ls-sol}
\phi_\mu(x^k+\theta^jP_kd^k)<\phi_{\mu}(x^k)-\eta\sqrt{\epsilon}\theta^{2j}\|d^k\|^2;
\end{equation}
\Else\ \{d$\_$type=NC\}
\State Find $\alpha_k=\theta^{j_k}$, where $j_k$ is the smallest nonnegative integer $j$ such that
\begin{equation}\label{ls-nc}
\phi_\mu(x^k+\theta^jP_kd^k)<\phi_{\mu}(x^k)-\eta\theta^{2j}\|d^k\|^3/2;
\end{equation}
\EndIf
\State $x^{k+1}=x^k+\alpha_kP_kd^k$;
\EndFor
\end{algorithmic}
}
\end{algorithm}

\begin{remark}
\begin{enumerate}[{\rm (i)}]
\item Though Algorithm \ref{alg:BNCG} finds a stochastic $(\epsilon,\sqrt{\epsilon})$-SOSP  of \eqref{conic-prob}, such a point is in fact also a deterministic $\epsilon$-FOSP of \eqref{conic-prob}, that is, it  satisfies \eqref{inexact-feasible}-\eqref{inexact-1st-order-2} deterministically.
\item Algorithm \ref{alg:BNCG} can be easily modified to suit some other needs. In particular, if one is only interested in finding an $\epsilon$-FOSP of \eqref{conic-prob}, it suffices to remove from Algorithm \ref{alg:BNCG} the parts related to Algorithm \ref{pro:meo}. In addition, if one is interested in finding a deterministic $(\epsilon,\sqrt{\epsilon})$-SOSP of \eqref{conic-prob}, it is sufficient to replace Algorithm \ref{pro:meo} by a deterministic oracle for estimating the minimum eigenvalue of a real symmetric matrix.
\item It is worth noting that Algorithm \ref{alg:BNCG} uses a hybrid line search criterion inspired by \cite[Algorithm~1]{XW21pN}, which is a combination of the quadratic descent criterion \eqref{ls-sol} and the cubic descent criterion \eqref{ls-nc}. In contrast, the Newton-CG type of methods in \cite{OW21,ROW20} always use a cubic descent criterion regardless of the type of search directions. As a benefit of the hybrid line search criteria, the iteration and operation complexity of Algorithm~\ref{alg:BNCG} has a quadratic dependence on the Lipschitz constant of $\nabla^2 f$ (see Theorems~\ref{iter-complexity} and \ref{thm:operation-cmplxty} below), which is superior to the cubic dependence achieved by the methods in  \cite{OW21,ROW20} for solving problems \eqref{unconstr-opt} and \eqref{boxconstr-prob}, respectively.
\end{enumerate}
\end{remark}

\section{Complexity results}\label{sec:complexity}

In this section we establish iteration and operation complexity results for the Newton-CG based barrier method, namely, Algorithm \ref{alg:BNCG}.

Recall that the cone $\cK$ is assumed to be equipped with a $\vartheta$-logarithmically homogeneous self-concordant barrier function $B$ for some $\vartheta \ge 1$. We now make some additional assumptions that will be used throughout this section. 
	
\begin{assumption} \label{main-assump}
\begin{enumerate}[{\rm (a)}]
\item There exist $\bmu \ge \mu$ and $\underline \phi\in\bR$ such that 
\beqa
&& \phi_\tmu(x) \ge \underline \phi\ \quad \forall \tmu \in (0,\bmu], x\in\Omega^{\rm o}, \label{phi-lwbd}\\ [4pt]
&& \cS = \underset{\tmu \in (0,\bmu]}{\bigcup}\{x\in\Omega^{\rm o}: \phi_\tmu(x) \le \phi_\tmu(x^0) \} \ \text{is bounded}, \label{level-bdd}
\eeqa
where $\Omega^{\rm o}$ is defined in Section \ref{intro}, $x^0\in\Omega^{\rm o}$ is the initial point of Algorithm \ref{alg:BNCG}, $\mu$ is given in Algorithm~\ref{alg:BNCG}, and $\phi_\tmu$ is given in \eqref{barrier-prob}.
\item There exists $L_H>0$ such that 
\begin{equation}\label{hes-Lip}
\|\nabla^2 f(y)-\nabla^2 f(x)\|^*_x \le L_H \|x-y\|_x\ \quad \forall x\in \cS, y\in\{y:\|y-x\|_x \le \beta \},
\end{equation}
where $\cS$ is given in \eqref{level-bdd}, and $\beta\in (0,1)$ is an input of Algorithm \ref{alg:BNCG}.
\item The quantities $U_g, U_H$ are finite, where 
\beq \label{Ugh}
U_g:=\sup\limits_{x\in\cS} \|\nabla f(x)\|^*_x, \quad U_H:=\sup\limits_{x\in\cS}\|\nabla^2 f(x)\|^*_x.
\eeq	
\end{enumerate}
\end{assumption}

We now make some remarks about Assumption \ref{main-assump}. 
\bi
\item[(i)] Assumption \ref{main-assump}(a) is reasonable. In particular, the assumption in \eqref{phi-lwbd} means that the barrier problem \eqref{barrier-prob} is uniformly bounded below whenever the barrier parameter is no larger than $\bmu$. It usually holds for the problems for which the barrier method converges. On the other hand, in case that \eqref{phi-lwbd} fails to hold, one can instead solve a perturbed counterpart of  \eqref{conic-prob}: 
\begin{equation}\label{conic-pert}
\min_x \{f(x)+\sigma \|x\|^2:Ax=b, x\in\mathcal{K}\}
\end{equation}
for some $\sigma>0$. It can be shown that a desired approximate FOSP and SOSP of \eqref{conic-prob} can be found by solving \eqref{conic-pert} with a sufficiently small $\sigma$.
Moreover, \eqref{phi-lwbd} with $f(x)$ being replaced by $f(x)+\sigma\|x\|^2$ holds for \eqref{conic-pert}. Indeed, let $\bmu>0$ be arbitrarily chosen and $f^*$ be the optimal value of \eqref{conic-prob}. Then for all $\tmu \in (0,\bmu]$ and $x\in\Omega^{\rm o}$, one has
\begin{align*}
f(x)+\sigma \|x\|^2 +\tmu B(x) & \geq  f^* + \min_{z\in\Omega^{\rm o}} \{\sigma \|z\|^2 +\tmu B(z)\} \geq  f^* + \tmu \min_{z\in\Omega^{\rm o}} \{(\sigma/\bmu) \|z\|^2 + B(z)\} \\
& \geq f^* - \bmu \,|\min_{z\in\Omega^{\rm o}} \{(\sigma/\bmu) \|z\|^2 + B(z)\}| >-\infty,
\end{align*}
where the last inequality is due to the strong convexity of $(\sigma/\bmu) \|z\|^2 + B(z)$. Hence, the assumption in \eqref{phi-lwbd} holds for \eqref{conic-pert} as desired.

Besides, the assumption in \eqref{level-bdd}  clearly holds if $\Omega^{\rm o}$ is bounded, which is assumed in \cite{HLY19} for $\cK=\bR^n_{+}$. Also, it can be shown that 
$\cS \subseteq \cS_1 \cup \cS_2$, where 
\[
\ba{l}
\cS_1 = \{x\in\Omega^{\rm o}: f(x) \le f(x^0)+\bmu+2\bmu [B(x^0)]_+,  B(x) \ge -1-[B(x^0)]_+\}, \\ [5pt]
\cS_2 = \left\{x\in\Omega^{\rm o}: \frac{f(x)}{-B(x)} \le \frac{[f(x^0)]_+}{1+[B(x^0)]_+}+2\bmu, B(x) \le -1-[B(x^0)]_+\right\},
\ea
\]
and $[t]_+=\max\{t,0\}$ for all $t\in\bR$. Thus the assumption in \eqref{level-bdd} holds if $\cS_1$ and $\cS_2$ are bounded,  
which, for example, holds for $f(x)=\ell(x)+\sum_{i=1}^nx_i^p$, $B(x)=-\sum_{i=1}^n\ln x_i$ and $\cK=\bR^n_+$ that are studied in \cite{BCY2015}, where $\ell:\bR^n \to \bR_+$ is a loss function and $p>0$.

\item[(ii)] Assumption \ref{main-assump}(b) means that $\nabla^2 f$ is locally Lipschitz continuous in $\cS$ with respect to the local norms. It holds if $\nabla^2 f$ is globally Lipschitz continuous in $\rmint\cK$, which is implicitly assumed in \cite{OW21} for the case where $A=0$, $b=0$ and $\cK=\bR^n_+$. Compared to the usual global Lipschitz continuity assumption on $\nabla^2 f$ in $\rmint\cK$, Assumption \ref{main-assump}(b) is generally weaker and holds for a broader class of problems. For example, Assumption \ref{main-assump}(b) holds for the problem with $f(x)=\sum_i x^p_i$ and $\cK=\bR^n_+$ for some $p\in (0,1)$, while $\nabla^2 f$ is not globally Lipschitz continuous in $\rmint\cK$. 

\item[(iii)] Since $\cS$ is assumed to be bounded, Assumption  \ref{main-assump}(c) can easily hold under some additional yet mild assumption on $f$. For example, by the boundedness of $\cS$ and Theorem \ref{bound-nabla-2B}, one can see that Assumption  \ref{main-assump}(c) holds if $\nabla f$ and $\nabla^2 f$ are continuous in $\cK$. In addition, one can verify  that Assumption \ref{main-assump}(c) also holds if $f$ and $\nabla f$ are locally Lipschitz continuous in $\cS$ with respect to the local norms, that is,
\[
\ba{l}
|f(y)-f(x)| \le U_g \|x-y\|_x\ \quad \forall x\in \cS, y\in\{y: \|y-x\|_x \le \beta \},  \\ [6pt]
\|\nabla f(y)-\nabla f(x)\|^*_x \le U_H \|x-y\|_x\ \quad \forall x\in \cS, y\in\{y: \|y-x\|_x \le \beta \}.
\ea
\]
These relations hold for a broad class of problems, such as the one with $f(x)=\sum_i x^p_i$ and $\cK=\bR^n_+$ for some $p\in (0,1)$. Note that Assumption \ref{main-assump}(c) is generally weaker than the one imposed in \cite{OW21} that $\nabla f$ and $\nabla^2 f$ are bounded in some level set of $f$, which, for example, does not hold for $f(x)=\sum_i x^p_i$ and $\cK=\bR^n_+$ for some $p\in (0,1)$. 

\item[(iv)] As will be shown in Lemma~\ref{feasibility}, each iterate $x^k$ of Algorithm \ref{alg:BNCG} lies in $\cS$. By this and Assumption~\ref{main-assump}(c), one can see that
\beq \label{U-bounds}
\|\nabla f(x^k)\|^*_{x^k}\le U_g, \quad \|\nabla^2 f(x^k)\|^*_{x^k} \le U_H.
\eeq
\ei

In addition, as a consequence of Assumption~\ref{main-assump}(b), the following two inequalities hold, which will play a crucial role in our subsequent analysis.
\begin{lemma}
Under Assumption~\ref{main-assump}(b), the following inequalities hold:
\begin{equation}\label{ineq:smooth-ppty1}
\|\nabla f(y)-\nabla f(x)-\nabla^2 f(x)(y-x)\|^*_x \le \frac{1}{2}L_{H}\|y-x\|^2_x\ \quad \forall x\in \cS,y\in\{y: \|y-x\|_x \le \beta\}, 
\end{equation}
\vspace{-4mm}
\begin{equation}\label{ineq:smooth-ppty2}
f(y)\le f(x) + \nabla f(x)^T(y-x) + \frac{1}{2} (y-x)^T\nabla^2 f(x)(y-x)+\frac{1}{6}L_{H}\|y-x\|^3_x\ \quad \forall x\in \cS,y\in\{y: \|y-x\|_x \le \beta\},
\end{equation}
where $\cS$ and $L_H$ are given in \eqref{level-bdd} and \eqref{hes-Lip}, respectively.
\end{lemma}
	
\begin{proof}
Fix any $x\in\cS$ and $y\in\{y: \|y-x\|_x\le\beta\}$.  One has
\[
\ba{lcl}
\|\nabla f(y)-\nabla f(x)-\nabla^2 f(x)(y-x)\|^*_x &=& \|\int^1_0 [\nabla^2 f(x+t(y-x))-\nabla^2 f(x)]dt(y-x)\|^*_x \\ [5pt]
&\overset{\eqref{M-norm}}{\le}&   \|\int^1_0 [\nabla^2 f(x+t(y-x))-\nabla^2 f(x)]dt\|^*_x\cdot \|y-x\|_x  \\ [5pt]
&\le&  \int^1_0 \|\nabla^2 f(x+t(y-x))-\nabla^2 f(x)\|^*_xdt \cdot \|y-x\|_x \\ [5pt]
&\overset{\eqref{hes-Lip}}{\le}&  L_H\int^1_0 \|t(y-x)\|_xdt \cdot \|y-x\|_x = \frac{1}{2}L_{H}\|y-x\|^2_x,
\ea
\]
and hence \eqref{ineq:smooth-ppty1} holds. We next prove \eqref{ineq:smooth-ppty2}. Indeed, one has
\[
\ba{l}
f(y) - f(x) - \nabla f(x)^T(y-x) - \frac{1}{2}(y-x)^T\nabla^2 f(x)(y-x)  \\
= \langle \int^1_0 [\nabla f(x+t(y-x))-\nabla f(x) -\nabla^2 f(x)t(y-x)]dt, y-x \rangle  \\ [5pt]
\le \| \int^1_0 [\nabla f(x+t(y-x))-\nabla f(x) -\nabla^2 f(x)t(y-x)]dt\|^*_x \cdot \|y-x\|_x  \\ [5pt]
\le \int^1_0 \|\nabla f(x+t(y-x))-\nabla f(x) -\nabla^2 f(x)t(y-x)\|^*_xdt \cdot \|y-x\|_x  \\ [5pt]
\overset{\eqref{ineq:smooth-ppty1}}{\le}  \frac{1}{2}L_{H} \int^1_0 \|t(y-x)\|^2_xdt \cdot \|y-x\|_x   = \frac{1}{6}L_{H}\|y-x\|^3_x,
\ea
\]
where the last inequality follows from \eqref{ineq:smooth-ppty1} with $y$ replaced by $x+t(y-x)$ for $t\in[0,1]$.
\end{proof}

\subsection{Iteration complexity}
In this subsection we establish iteration complexity results for Algorithm \ref{alg:BNCG} for solving problem \eqref{conic-prob}. Before proceeding, we establish several lemmas that will be used later.

The following lemma shows that all the iterates generated by Algorithm \ref{alg:BNCG} belong to the set $\cS$.

\begin{lemma} \label{feasibility}
Let $\{x^k\}_{k\in \bbK}$ be all the iterates generated by Algorithm \ref{alg:BNCG}, where $\bbK$ is a subset of consecutive nonnegative integers starting from $0$. Then $x^k \in \cS$ for every $k\in\bbK$, where $\cS$ is given in \eqref{level-bdd}.
\end{lemma}

\begin{proof}
We prove this lemma by induction. By the choice of $x^0$, one knows that $x^0\in \Omega^{\rm o}$, and hence $x^0\in\cS$ due to \eqref{level-bdd}. Suppose that $x^k\in\cS$ is generated at iteration $k$ of Algorithm \ref{alg:BNCG}, and moreover, $x^{k+1}$ is generated at iteration $k+1$. We now show that $x^{k+1}\in\cS$. Indeed, notice from Algorithm \ref{alg:BNCG} that $x^{k+1}=x^k+\alpha_k P_kd^k$ with $\alpha_k \in (0,1]$ and $d^k$ given in one of \eqref{dk-nc}-\eqref{dk-2nd-nc}. It follows from 
\eqref{dk-nc}-\eqref{dk-2nd-nc} that $\|Q_kd^k\|\le \beta$. By these, \eqref{Mk} and \eqref{Pk},  one has that 
\beq \label{xk-change}
\|x^{k+1}-x^k\|_{x^k} = \alpha_k \|P_kd^k\|_{x^k}   \le \|P_kd^k\|_{x^k} \overset{\eqref{Pk}}{=} \|M_kQ_kd^k\|_{x^k} \overset{\eqref{Mk}}{=} \|Q_kd^k\| \le \beta.
\eeq
In view of $x^k\in\cS$ and \eqref{level-bdd}, one can see that $x^k\in \Omega^0$. Hence,  
$Ax^k=b$ and $x^k\in\rmint\cK$. Using \eqref{xk-change}, $x^k\in\rmint \cK$, $\beta<1$ and Lemma \ref{lem:barrier-property}(iii), we obtain that $x^{k+1}\in\rmint\cK$. In addition, it follows from \eqref{Pk} that 
\[
AP_kd^k=A[M_k-M_kM_k^TA^T(AM_kM_k^TA^T)^{-1}AM_k]d^k=0,
\]
which, together with $Ax^k=b$ and $x^{k+1}=x^k+\alpha_k P_kd^k$, implies that $Ax^{k+1}=b$. 
It follows that $x^{k+1}\in\Omega^{\rm o}$. Observe from Algorithm \ref{alg:BNCG} that 
$\{\phi_{\mu}(x^k)\}_{k\in\bbK}$ is descent, and hence $\phi_{\mu}(x^{k+1}) \le \phi_\mu(x^0)$. 
By this, $x^{k+1}\in\Omega^{\rm o}$, $\mu \le \bmu$ and \eqref{level-bdd}, one can conclude that $x^{k+1}\in\cS$, and hence the induction is completed.
\end{proof}

The lemma below states some properties of the direction $d^k$ arising in Algorithm \ref{alg:BNCG} that results from applying Algorithm \ref{alg:capped-CG} to \eqref{linsys} with $H=P_k^T\nabla^2\phi_\mu(x^k)P_k$, $\varepsilon=\sqrt{\epsilon}$, $g=P_k^T\nabla\phi_\mu(x^k)$. Its proof is similar to the ones in \cite[Lemma~7]{OW21} and \cite[Lemma~3]{ROW20} and thus omitted here.

\begin{lemma}\label{lem:SOL-NC-ppty}
Suppose that the direction $d^k$ results from the output $\widehat{d}^k$ of Algorithm \ref{alg:capped-CG}  with a type specified in d$\_$type  at some iteration $k$ of Algorithm \ref{alg:BNCG}. Let $Q_k$ and $P_k$ be given in \eqref{Qk} and \eqref{Pk}, respectively. 
Then the following statements hold.
\begin{enumerate}[{\rm (i)}]
\item If d$\_$type=SOL, then  $d^k$ satisfies
\beqa
&\sqrt{\epsilon}\|d^k\|^2\le (d^k)^T\left(P_k^T\nabla^2\phi_{\mu}(x^k)P_k+2\sqrt{\epsilon}I\right)d^k,\label{SOL-ppty-1}\\
&\|d^k\|\le1.1 \epsilon^{-1/2}\|P_k^T\nabla\phi_{\mu}(x^k)\|,\label{SOL-ppty-2}\\
&(d^k)^TP_k^T\nabla\phi_{\mu}(x^k)=-\gamma_k(d^k)^T\left(P_k^T\nabla^2\phi_{\mu}(x^k)P_k+2\sqrt{\epsilon}I\right)d^k,\label{SOL-ppty-3}
\eeqa
where $\gamma_k=\max\{\|Q_k\widehat{d}^k\|/\beta,1\}$. If $\|Q_k\widehat{d}^k\|\le\beta$, then $d^k$ also satisfies
\begin{equation}\label{ppty-residual-bound}
\|(P_k^T\nabla^2\phi_{\mu}(x^k)P_k+2\sqrt{\epsilon}I)d^k+P_k^T\nabla\phi_{\mu}(x^k)\|\le\sqrt{\epsilon}\zeta\|d^k\|/2.
\end{equation}
\item If d$\_$type=NC, then  $d^k$ satisfies $(d^k)^TP_k^T\nabla\phi_{\mu}(x^k)\le0$ and
\[
\frac{(d^k)^TP_k^T\nabla^2\phi_{\mu}(x^k)P_kd^k}{\|d^k\|^2}\le-\|d^k\|\le-\sqrt{\epsilon}.
\]
\end{enumerate}
\end{lemma}
	
The next lemma considers the case where the direction $d^k$ in Algorithm \ref{alg:BNCG}
 results from the output of Algorithm \ref{alg:capped-CG} with d$\_$type=SOL, and moreover, the unit step length is accepted by the line search procedure. For this case, it will be shown that $\|d^k\|$ cannot be too small or the next iterate $x^{k+1}$ is an approximate first-order stationary point. 

\begin{lemma}\label{lem:SOL-step-lowerbd1}
Suppose that the direction $d^k$ results from the output of Algorithm \ref{alg:capped-CG} with d$\_$type=SOL at some iteration $k$ of Algorithm \ref{alg:BNCG}, and the unit step length is accepted by the line search procedure, that is, $x^{k+1}=x^k+P_kd^k$. Then we have $\|d^k\|\ge c_d\sqrt{\epsilon}$ or 
\begin{equation}\label{optcond:1st-iterate-k+1}
\|\nabla f(x^{k+1})+A^T\lambda_{k+1}^{(2)}+\mu\nabla B(x^k)\|_{x^{k+1}}^*\le (1-\beta)\mu,
\end{equation}
where 
\beq  \label{c_d}
{c_d=\frac{(1-\beta)^3}{(L_H +\zeta+4)[(1-\beta)^2+\sqrt{\vartheta}]+1-\beta}}, \qquad
\lambda_{k+1}^{(2)}=R_k(\nabla^2 f(x^k)P_kd^k + \nabla\phi_{\mu}(x^k)),
\eeq
and $P_k$ and $R_k$ are given in \eqref{Pk} and \eqref{Rk}, respectively.
\end{lemma}
	
\begin{proof}
Since $d^k$ results from the output $\widehat{d}^k$ of Algorithm \ref{alg:capped-CG} with d$\_$type=SOL,  it follows from Algorithm \ref{alg:BNCG} that \eqref{dk-sol} holds for $d^k$ and $\widehat{d}^k$.  In addition, one can observe from \eqref{xk-change} that $\|x^{k+1}-x^k\|_{x^k} \le \beta$. Also, by Lemma \ref{feasibility}, one has that $x^k\in \cS$. Hence, \eqref{ineq:smooth-ppty1} holds for $x=x^k$ and $y=x^{k+1}$.  Let $Q_k$ be given in \eqref{Qk}. We now divide the rest of the proof into two separate cases below.

Case 1) $\|Q_k\widehat{d}^k\|\ge\beta$. It then follows from \eqref{dk-sol} that $d^k=\beta \widehat{d}^k/\|Q_k\widehat{d}^k\|$. In addition, one can observe from Algorithm \ref{alg:BNCG} that $\beta \ge \sqrt{\epsilon}$. By these and $\|Q_k\|=1$, we have
\beq \label{dk-lwbd}
\sqrt{\epsilon}\le\beta=\|Q_kd^k\|\le\|Q_k\|\|d^k\|=\|d^k\|.
\eeq
Notice from \eqref{c_d} that $c_d\le1$, which together with \eqref{dk-lwbd} implies that $\|d^k\|\ge c_d\sqrt{\epsilon}$ and thus the conclusion holds.

Case 2) $\|Q_k\widehat{d}^k\|<\beta$.	 Notice that if $\|d^k\| \ge c_d\sqrt{\epsilon}$, the conclusion of this lemma holds. Hence, it suffices to consider the case where $\|Q_k\widehat{d}^k\|<\beta$  and $\|d^k\|<c_d\sqrt{\epsilon}$. 
We next show that  \eqref{optcond:1st-iterate-k+1} holds in this case. 
To this end, suppose for the rest of the proof that $\|Q_k\widehat{d}^k\|<\beta$  and $\|d^k\|<c_d\sqrt{\epsilon}$. 
Since d$\_$type=SOL and $\|Q_k\widehat{d}^k\|<\beta$, one can see from Lemma \ref{lem:SOL-NC-ppty}(i) that \eqref{ppty-residual-bound} holds for $d^k$.  By \eqref{Pk},  \eqref{ppty-residual-bound}  and the definition of $\phi_\mu$, one has that
\vspace{-3mm}
\begin{eqnarray}
\frac{1}{2}\sqrt{\epsilon}\zeta\|d^k\|&\overset{\eqref{ppty-residual-bound}}{\ge}&\left\|\left(P_k^T\nabla^2\phi_{\mu}(x^k)P_k+2\sqrt{\epsilon}I\right)d^k+P_k^T\nabla\phi_{\mu}(x^k)\right\|\nonumber\\
&=& \left\|\left(P_k^T(\nabla^2f(x^k)+\mu \nabla^2 B(x^k))P_k+2\sqrt{\epsilon}I\right)d^k+P_k^T\nabla\phi_{\mu}(x^k)\right\|\nonumber\\
&=&\left\|P_k^T\Big(\nabla^2f(x^k)P_kd^k+\nabla\phi_{\mu}(x^k)\Big)+\mu P_k^T\nabla^2 B(x^k)P_kd^k+2\sqrt{\epsilon} d^k\right\|\nonumber\\
&\ge&\left\|P_k^T\Big(\nabla^2f(x^k)P_kd^k+\nabla\phi_{\mu}(x^k)\Big)\right\|-\mu\|P_k^T\nabla^2 B(x^k)P_kd^k\|-2\sqrt{\epsilon}\|d^k\|\nonumber\\
&\overset{\eqref{Pk}}{\ge}&\left\|P_k^T\Big(\nabla^2f(x^k)P_kd^k+\nabla\phi_{\mu}(x^k)\Big)\right\|-\mu\|Q_k\|^2\|M_k^T\nabla^2 B(x^k)M_k\|\|d^k\|-2\sqrt{\epsilon}\|d^k\|\nonumber\\
&=&\left\|P_k^T\Big(\nabla^2f(x^k)P_kd^k+\nabla\phi_{\mu}(x^k)\Big)\right\|-\mu\|d^k\|-2\sqrt{\epsilon}\|d^k\|,\label{ineq:residual-bound-1}
\end{eqnarray}
where the first equality is due to the definition of $\phi_\mu$, the second inequality is due to the triangle inequality, and the last equality follows from the fact that $\|Q_k\|=1$ and $M_k^T\nabla^2 B(x^k)M_k=I$ (due to \eqref{Mk}). 
Using \eqref{ineq:dual-local-norm-iterate-ppty}, \eqref{Mk},  \eqref{eq:relation-Pk-Rk}, $\|x^{k+1}-x^k\|_{x^k} \le \beta$, and the definition of $\lambda_{k+1}^{(2)}$, we obtain that
\begin{eqnarray}
&&\left\|P_k^T\Big(\nabla^2f(x^k)P_kd^k+\nabla\phi_{\mu}(x^k)\Big)\right\|\nonumber\\
&&\overset{\eqref{eq:relation-Pk-Rk}}{=}\left\|M_k^T\Big(\nabla^2f(x^k)P_kd^k+\nabla\phi_{\mu}(x^k)\Big)+M_k^TA^TR_k\Big(\nabla^2f(x^k)P_kd^k+\nabla\phi_{\mu}(x^k)\Big)\right\|\nonumber\\
&&=\left\|M_k^T\Big(\nabla^2f(x^k)P_kd^k+\nabla\phi_{\mu}(x^k)+A^T\lambda^{(2)}_{k+1}\Big)\right\|\nonumber\\
&&\overset{\eqref{Mk}}{=}\left\|\nabla^2f(x^k)P_kd^k+\nabla f(x^k)+A^T\lambda^{(2)}_{k+1}+\mu\nabla B(x^k)\right\|_{x^k}^*\nonumber\\
&&\overset{\eqref{ineq:dual-local-norm-iterate-ppty}}{\ge}(1-\beta)\left\|\nabla^2f(x^k)P_kd^k+\nabla f(x^k)+A^T\lambda_{k+1}^{(2)}+\mu\nabla B(x^k)\right\|_{x^{k+1}}^*,\label{ineq:residual-bound-2}
\end{eqnarray}
where the second equality follows from the definition of $\lambda_{k+1}^{(2)}$.
Combining \eqref{ineq:residual-bound-1} with \eqref{ineq:residual-bound-2} yields
\begin{equation}
\left\|\nabla^2f(x^k)P_kd^k+\nabla f(x^k)+A^T\lambda^{(2)}_{k+1}+\mu\nabla B(x^k)\right\|_{x^{k+1}}^*\le\frac{(\zeta\sqrt{\epsilon}+2\mu+4\sqrt{\epsilon})\|d^k\|}{2(1-\beta)}.\label{ineq:residual-bound-3}
\end{equation}
 In addition, by $\|Q_k\|=1$ and \eqref{xk-change}, one has that 
\beq \label{pd-norm}
\|P_kd^k\|_{x^k} =\|Q_kd^k\| \le \|d^k\|.
\eeq
Also, notice from \eqref{c_d} and Algorithm \ref{alg:BNCG} that $0<c_d<1$ and {$\mu=(1-\beta)\epsilon/[2((1-\beta)^2+\sqrt{\vartheta})]$}, respectively. Using these, \eqref{ineq:dual-local-norm-iterate-ppty},  \eqref{ineq:smooth-ppty1},  \eqref{ineq:residual-bound-3}, \eqref{pd-norm}, $\|x^{k+1}-x^k\|_{x^k} \le \beta$, and $\|d^k\|<c_d\sqrt{\epsilon}$, we have that
\begin{eqnarray}
&&\hspace{-.25in}\left\|\nabla f(x^{k+1})+A^T\lambda^{(2)}_{k+1}+\mu\nabla B(x^k)\right\|_{x^{k+1}}^*\nonumber\\
&&\hspace{-.25in}\le\left\|\nabla f(x^{k+1})-\nabla^2 f(x^k)P_kd^k - \nabla f(x^k)\right\|_{x^{k+1}}^* + \left\|\nabla^2f(x^k)P_kd^k+\nabla f(x^k)+A^T\lambda^{(2)}_{k+1}+\mu\nabla B(x^k)\right\|_{x^{k+1}}^*\nonumber\\
&&\hspace{-.25in}\overset{\eqref{ineq:dual-local-norm-iterate-ppty}}{\le} (1-\beta)^{-1} \left\|\nabla f(x^{k+1})-\nabla^2 f(x^k)P_kd^k - \nabla f(x^k)\right\|^*_{x^k} + \left\|\nabla^2f(x^k)P_kd^k+\nabla f(x^k)+A^T\lambda^{(2)}_{k+1}+\mu\nabla B(x^k)\right\|_{x^{k+1}}^*\nonumber\\
&&\hspace{-.25in}\overset{\eqref{ineq:smooth-ppty1}\eqref{ineq:residual-bound-3}}{\le}\frac{L_H \|P_kd^k\|^2_{x^k}}{2(1-\beta)}+\frac{(\zeta\sqrt{\epsilon}+2\mu+4\sqrt{\epsilon})\|d^k\|}{2(1-\beta)}\ \overset{\eqref{pd-norm}}{\le} \ \frac{L_H \|d^k\|^2}{2(1-\beta)}+\frac{(\zeta\sqrt{\epsilon}+2\mu+4\sqrt{\epsilon})\|d^k\|}{2(1-\beta)} \nonumber\\
&&\hspace{-.25in}< \ \frac{L_H c_d\epsilon}{2(1-\beta)}+\frac{(\zeta+4)c_d\epsilon}{2(1-\beta)}+\frac{c_d\mu}{1-\beta}
{ \ =\ \frac{(L_H +\zeta+4)[(1-\beta)^2+\sqrt{\vartheta}]+1-\beta}{(1-\beta)^2}c_d\mu\ = \ (1-\beta)\mu,}\nonumber
\end{eqnarray}
where the first inequality follows from the triangle inequality, the last inequality uses $\|d^k\|<c_d\sqrt{\epsilon}$,  $0<c_d<1$ and $\epsilon\le 1$, the first equality uses {$\mu=(1-\beta)\epsilon/[2((1-\beta)^2+\sqrt{\vartheta})]$}, and the last equality follows from the definition of $c_d$. Hence, \eqref{optcond:1st-iterate-k+1} holds as desired.
\end{proof}

The following lemma shows that if the direction $d^k$ in Algorithm \ref{alg:BNCG}  results from the output of Algorithm \ref{alg:capped-CG} with d$\_$type=SOL, then the associated step length $\alpha_k$ is well-defined, and moreover, the next iterate $x^{k+1}$ is an approximate first-order stationary point or 
$\phi_\mu(x^k)-\phi_\mu(x^{k+1})$ cannot be too small.

\begin{lemma}\label{lem:SOL-step-lowerbd2}	
Suppose that the direction $d^k$ results from the output of Algorithm \ref{alg:capped-CG} with d$\_$type=SOL at some iteration $k$ of Algorithm \ref{alg:BNCG}. 
Then the following statements hold.
\begin{enumerate}[{\rm (i)}]
\item The step length $\alpha_k$ is well-defined, and moreover, 
\beq \label{alphak-lwbd}
{
\alpha_k \ge \min\left\{1,\frac{\sqrt{6(1-\beta)(1-\eta)\epsilon}\theta}{\sqrt{1.1[L_H(1-\beta)+1/2](U_g+\mu \sqrt{\vartheta})}}\right\},}
\eeq
where $U_g$ is defined in \eqref{Ugh}.
\item 
The relation \eqref{optcond:1st-iterate-k+1} holds for $(x^{k+1},\lambda^{(2)}_{k+1})$ or 
{$\phi_\mu(x^k)-\phi_\mu(x^{k+1})> c_{\rm sol}\epsilon^{3/2}$ holds},
where   
\beq \label{csol}
{
c_{\text{sol}}=\eta\min\left\{c_d^2,\left[\frac{6(1-\beta)(1-\eta)\theta}{L_H(1-\beta)+1/2}\right]^2\right\},
}
\eeq
and $\lambda^{(2)}_{k+1}$ and $c_d$ are given in \eqref{c_d}.
\end{enumerate}
\end{lemma}
	
\begin{proof}
For notational convenience, let $H=P_k^T\nabla^2\phi_{\mu}(x^k)P_k$ and $g=P_k^T\nabla\phi_{\mu}(x^k)$. Since  d$\_$type=SOL, it follows from Lemma \ref{lem:SOL-NC-ppty}(i) that \eqref{SOL-ppty-1}, \eqref{SOL-ppty-2} and \eqref{SOL-ppty-3} hold. Also, by $\vartheta\ge 1$, {$\epsilon<1$} and $0<\beta<1$, one has that  
\beq \label{mu-bdd}
{
\mu = \frac{(1-\beta)\epsilon}{2[(1-\beta)^2+\sqrt{\vartheta}]} \le \frac{(1-\beta)\epsilon}{2[(1-\beta)^2+1]} \le \frac{\epsilon}{4}< \frac{1}{4}.}
\eeq
In addition, by Lemma \ref{feasibility}, one has that $x^k\in \cS$. Also, one can observe from \eqref{xk-change} that $\|\theta^jP_kd^k\|_{x^k} \le \beta$ for all $j \ge 0$.  Hence, \eqref{ineq:smooth-ppty2} holds for $x=x^k$ and $y=x^k+\theta^jP_kd^k$ for all $j\ge 0$.

We are now ready to prove statement (i). If {\eqref{ls-sol}} holds for $j=0$, then the line search procedure chooses the unit step length, i.e., $\alpha_k=1$, and hence statement (i) holds. We now suppose that {\eqref{ls-sol}} fails for $j=0$. Let us consider all $j\ge0$ that violate \eqref{ls-sol}. For any such $j$, by using  \eqref{ineq:smooth-ppty2}, \eqref{SOL-ppty-1}, \eqref{SOL-ppty-3}, \eqref{pd-norm}, Lemma \ref{lem:barrier-property}(v), and $\mu<1/4$, one has that
\[
\ba{l}
{-\eta\sqrt\epsilon\theta^{2j}\|d^k\|^2} \le\phi_{\mu}(x^k+\theta^jP_kd^k)-\phi_{\mu}(x^k)
=f(x^k+\theta^jP_kd^k)-f(x^k)+\mu[B(x^k+\theta^jP_kd^k)-B(x^k)]\\[8pt]
\le\theta^j \nabla f(x^k)^T P_kd^k + \frac{\theta^{2j}}{2}(d^k)^TP_k^T\nabla^2 f(x^k)P_kd^k + \frac{L_H}{6}\theta^{3j}\|P_kd^k\|_{x^k}^3+\mu\theta^j\nabla B(x^k)^T P_kd^k \\[8pt] 
\quad+\frac{\mu\theta^{2j}}{2}(d^k)^TP_k^T\nabla^2 B(x^k)P_kd^k 
+ \frac{\mu}{3(1-\beta)}\theta^{3j}\|P_kd^k\|_{x^k}^3\\[8pt]
= \theta^j g^Td^k + \frac{\theta^{2j}}{2}(d^k)^THd^k + \frac{L_H(1-\beta)+2\mu}{6(1-\beta)}\theta^{3j}\|P_kd^k\|_{x^k}^3\\[8pt]
\overset{\eqref{SOL-ppty-3}}{=} -\theta^j\gamma_k (d^k)^T(H+2\sqrt{\epsilon}I)d^k + \frac{\theta^{2j}}{2}(d^k)^THd^k+  \frac{L_H(1-\beta)+2\mu}{6(1-\beta)}\theta^{3j}\|P_kd^k\|_{x^k}^3\\[8pt]
= -\theta^j\left(\gamma_k - \frac{\theta^j}{2} \right)(d^k)^T(H+2\sqrt{\epsilon}I)d^k - \theta^{2j}\sqrt{\epsilon}\|d^k\|^2+ \frac{L_H(1-\beta)+2\mu}{6(1-\beta)}\theta^{3j}\|P_kd^k\|^3_{x^k}\\[8pt]
\overset{\eqref{SOL-ppty-1}\eqref{pd-norm}}{\le} -\theta^j\left(\gamma_k - \frac{\theta^j}{2} \right) \sqrt{\epsilon}\|d^k\|^2 - \theta^{2j}\sqrt{\epsilon}\|d^k\|^2 + {\frac{L_H(1-\beta)+2\mu}{6(1-\beta)}}\theta^{3j}\|d^k\|^3 \\[5pt]
\le -\theta^j \gamma_k \sqrt{\epsilon}\|d^k\|^2 + {\frac{L_H(1-\beta)+1/2}{6(1-\beta)}}\theta^{3j}\|d^k\|^3,
\ea
\]
where the first inequality is due to the violation of \eqref{ls-sol}, the first equality follows from the definition of $\phi_\mu$, the second inequality uses \eqref{ineq:smooth-ppty2} and Lemma \ref{lem:barrier-property}(v), the second equality follows from $H=P_k^T\nabla^2\phi_{\mu}(x^k)P_k$ and $g=P_k^T\nabla\phi_{\mu}(x^k)$, and the last inequality is due to  $\mu<1/4$.  Using the last inequality above, {$\gamma_k=\max\{\|Q_k\widehat{d}^k\|/\beta,1\}\ge 1$, $\theta\in(0,1)$, $\eta\in(0,1)$} and the fact that $d^k\neq 0$ (see Theorem \ref{lem:capped-CG-cmplxity}), we obtain that
\beq \label{ineq:SOL-lowerbound-1}
{\frac{L_H(1-\beta)+1/2}{6(1-\beta)}\theta^{2j} \ge (\gamma_k-\eta\theta^j)\sqrt{\epsilon}\|d^k\|^{-1} \ge (1-\eta)\sqrt{\epsilon}\|d^k\|^{-1},}
\eeq
which, together with $\theta\in(0,1)$, implies that all $j\ge0$ that violate {\eqref{ls-sol}} must be bounded above. Hence, there does exist the smallest positive integer $j_k$ such that {\eqref{ls-sol}} holds for $j=j_k$, and thus $\alpha_k$ is well-defined. We next show that \eqref{alphak-lwbd} holds for $\alpha_k$. Indeed, we know from Lemma \ref{lem:barrier-property}(i), \eqref{Pk} and \eqref{U-bounds} that $\|\nabla B(x^k)\|_{x^k}^*=\sqrt{\vartheta}$, $P_k=M_kQ_k$, and $\|\nabla f(x^k)\|^*_{x^k}\le U_g$, respectively. By these, $\|Q_k\|=1$, \eqref{Mk} and  \eqref{SOL-ppty-2}, one has that
\beqa 
\|d^k\| &\overset{\eqref{SOL-ppty-2}}{\le}& 1.1\epsilon^{-1/2}\|P_k^T\nabla\phi_{\mu}(x^k)\|\le 1.1 \epsilon^{-1/2}(\|P_k^T\nabla f(x^k)\|+\mu\|P_k^T\nabla B(x^k)\|)\nonumber \\
&\le & 1.1 \epsilon^{-1/2}(\|M_k^T\nabla f(x^k)\|+\mu\|M_k^T\nabla B(x^k)\|) 
\overset{ \eqref{Mk}}{=} 1.1 \epsilon^{-1/2}(\|\nabla f(x^k)\|^*_{x^k}+\mu\|\nabla B(x^k)\|^*_{x^k}) \nonumber \\
&\le& 1.1 \epsilon^{-1/2}(U_g+\mu \sqrt{\vartheta}). \label{dk-bound}
\eeqa
Notice from Algorithm \ref{alg:BNCG} that $j=j_k-1$ violates {\eqref{ls-sol}} and hence 
\eqref{ineq:SOL-lowerbound-1} holds for $j=j_k-1$. By  $\alpha_k=\theta^{j_k}$ and \eqref{ineq:SOL-lowerbound-1} with $j=j_k-1$,  one has that 
\beq \label{alphak-lwbd1}
{
\alpha_k = \theta^{j_k} \ge \sqrt{\frac{6(1-\beta)(1-\eta)}{L_H(1-\beta)+1/2}}\theta\epsilon^{1/4}\|d^k\|^{-1/2},
}
\eeq
which together with \eqref{dk-bound} implies that \eqref{alphak-lwbd} holds.

We next prove statement (ii) by considering three separate cases. 
 
Case 1) $\alpha_k=1$ and $\|d^k\|<c_d\sqrt{\epsilon}$. It follows from Lemma \ref{lem:SOL-step-lowerbd1} that \eqref{optcond:1st-iterate-k+1} holds for $(x^{k+1},\lambda^{(2)}_{k+1})$.

Case 2) $\alpha_k=1$ and $\|d^k\| \ge c_d\sqrt{\epsilon}$. By these, {\eqref{ls-sol}} and \eqref{csol}, one has
{
\[
\phi_{\mu}(x^k)-\phi_\mu(x^{k+1}) > \eta\sqrt{\epsilon}\|d^k\|^2 \ge \eta c_d^2\epsilon^{3/2}\ge c_{\text{sol}}\epsilon^{3/2},
\]
}
and hence statement (ii) holds.

Case 3) $\alpha_k<1$. It implies that {\eqref{ls-sol}} fails for $j=0$. As seen from the proof of statement (i),  \eqref{ineq:SOL-lowerbound-1} holds for $j=0$ and \eqref{alphak-lwbd1} also holds. By setting $j=0$ in \eqref{ineq:SOL-lowerbound-1}, one has
\[
{\|d^k\|\ge\frac{6(1-\beta)(1-\eta)}{L_H(1-\beta)+1/2}\sqrt{\epsilon}.}
\]
Using this, {\eqref{ls-sol}}, \eqref{csol} and \eqref{alphak-lwbd1}, we obtain that
{
\[
\phi_\mu(x^k)-\phi_\mu(x^{k+1}) > \eta\sqrt{\epsilon}\alpha_k^2\|d^k\|^2 \ge \eta \frac{6(1-\beta)(1-\eta)\theta^2\epsilon}{L_H(1-\beta)+1/2}\|d^k\|\ge \eta\left[\frac{6(1-\beta)(1-\eta)\theta}{L_H(1-\beta)+1/2}\right]^2\epsilon^{3/2}\ge c_{\text{sol}}\epsilon^{3/2},
\]
}
and hence statement (ii) holds.
\end{proof}

The next lemma shows that if the direction $d^k$ in Algorithm \ref{alg:BNCG}  results from the output of Algorithm \ref{alg:capped-CG} with d$\_$type=NC, then the associated step length $\alpha_k$ is well-defined, and moreover, {the reduction on the function value of $\phi_\mu$, i.e., $\phi_\mu(x^k)-\phi_\mu(x^{k+1})$, cannot be too small.}

\begin{lemma}\label{lem:NC-step-lowerbd}	
Suppose that the direction $d^k$ results from the output of Algorithm \ref{alg:capped-CG} with d$\_$type=NC at some iteration $k$ of Algorithm \ref{alg:BNCG}. Let 
{
\beq \label{cnc}
c_{\text{nc}}=\frac{\eta}{2}\min\left\{1,\left[\frac{3(1-\beta)(1-\eta)\theta}{L_H(1-\beta)+1/2}\right]^2\right\}.
\eeq 
}
Then the following statements hold.
\begin{enumerate}[{\rm (i)}]
\item The step length $\alpha_k$ is well-defined, and moreover,
\begin{equation}\label{ak-nc-lwbd}
{
\alpha_k\ge\min\left\{1,\frac{3(1-\beta)(1-\eta)\theta}{L_H(1-\beta)+1/2}\right\}.}
\end{equation}
\item {$\phi_\mu(x^k)-\phi_\mu(x^{k+1})\ge c_{\text{nc}}\epsilon^{3/2}$} holds.
\end{enumerate}
\end{lemma}

\begin{proof}
For notational convenience, let $H=P_k^T\nabla^2\phi_{\mu}(x^k)P_k$ and $g=P_k^T\nabla\phi_{\mu}(x^k)$. Since  d$\_$type=NC, it then follows from Lemma \ref{lem:SOL-NC-ppty}(ii) that $(d^k)^Tg\le0$, $(d^k)^THd^k\le-\|d^k\|^3$, and $\|d^k\| \ge \sqrt{\epsilon}$.  In addition, by Lemma \ref{feasibility}, one has that $x^k\in \cS$. Also, one can observe from \eqref{xk-change} that $\|\theta^jP_kd^k\|_{x^k} \le \beta$ for all $j \ge 0$.  Hence, \eqref{ineq:smooth-ppty2} holds for $x=x^k$ and $y=x^k+\theta^jP_kd^k$ for all $j\ge 0$. Also, recall from \eqref{mu-bdd} that {$\mu<1/4$}.

We are now ready to prove statement (i). If {\eqref{ls-nc}} holds for $j=0$, then the line search procedure chooses the unit step length, i.e., $\alpha_k=1$, and hence statement (i) holds. We now suppose that {\eqref{ls-nc}} fails for $j=0$. Let us consider all $j\ge0$ that violate \eqref{ls-nc}. For any such $j$, by using \eqref{ineq:smooth-ppty2}, \eqref{pd-norm}, Lemma \ref{lem:barrier-property}(v), $(d^k)^Tg\le0$, $(d^k)^THd^k\le-\|d^k\|^3$, and {$\mu<1/4$}, one has that
\[
\ba{l}
{-\frac{\eta}{2}\theta^{2j}\|d^k\|^3} \le\phi_{\mu}(x^k+\theta^jP_kd^k)-\phi_{\mu}(x^k)
=f(x^k+\theta^jP_kd^k)-f(x^k)+\mu[B(x^k+\theta^jP_kd^k)-B(x^k)]\\[8pt]
\le\theta^j \nabla f(x^k)^T P_kd^k + \frac{\theta^{2j}}{2}(d^k)^TP_k^T\nabla^2 f(x^k)P_kd^k + \frac{L_H}{6}\theta^{3j}\|P_kd^k\|_{x^k}^3+\mu\theta^j\nabla B(x^k)^T P_kd^k \\[8pt] 
\quad+\frac{\mu\theta^{2j}}{2}(d^k)^TP_k^T\nabla^2 B(x^k)P_kd^k 
+ \frac{\mu}{3(1-\beta)}\theta^{3j}\|P_kd^k\|_{x^k}^3\\[8pt]
= \theta^j g^Td^k + \frac{\theta^{2j}}{2}(d^k)^THd^k +  \frac{L_H(1-\beta)+2\mu}{6(1-\beta)}\theta^{3j}\|P_kd^k\|_{x^k}^3\\[8pt]
\le -\frac{\theta^{2j}}{2} \|d^k\|^3 + {\frac{L_H(1-\beta)+1/2}{6(1-\beta)}}\theta^{3j}\|d^k\|^3,
\ea
\]
where the first inequality is due to the violation of \eqref{ls-nc}, the first equality follows from the definition of $\phi_\mu$, the second inequality uses \eqref{ineq:smooth-ppty2} and Lemma \ref{lem:barrier-property}(v), the second equality follows from $H=P_k^T\nabla^2\phi_{\mu}(x^k)P_k$ and $g=P_k^T\nabla\phi_{\mu}(x^k)$, and the last inequality follows from  \eqref{pd-norm}, $(d^k)^Tg\le0$, $(d^k)^THd^k\le-\|d^k\|^3$ and $\mu<1/4$. Using the last inequality above and the fact that $d^k\neq 0$ (see Theorem \ref{lem:capped-CG-cmplxity}), we obtain that
\begin{equation}\label{ineq:NC-iternumber-upperbd-1}
{\theta^j \ge \frac{3(1-\beta)(1-\eta)}{L_H(1-\beta)+1/2},}
\end{equation}
which, together with $\theta\in(0,1)$, implies that all $j\ge0$ that violate {\eqref{ls-nc}} must be bounded above. Hence, there does exist the smallest positive integer $j_k$ such that {\eqref{ls-nc}} holds for $j=j_k$, and thus $\alpha_k$ is well-defined. {We next prove \eqref{ak-nc-lwbd}.}
Indeed, notice from Algorithm \ref{alg:BNCG} that $j=j_k-1$ violates {\eqref{ls-nc}} and hence 
\eqref{ineq:NC-iternumber-upperbd-1} holds for $j=j_k-1$. By  $\alpha_k=\theta^{j_k}$ and \eqref{ineq:NC-iternumber-upperbd-1} with $j=j_k-1$,  one has that 
\[
{
\alpha_k = \theta^{j_k} \ge  \frac{3(1-\beta)(1-\eta)\theta}{L_H(1-\beta)+1/2},
}
\]
which {proves \eqref{ak-nc-lwbd} as desired}.

{
Statement (ii) immediately follows from \eqref{ls-nc}, \eqref{cnc}, \eqref{ak-nc-lwbd} and the fact that $\|d^k\| \ge \sqrt{\epsilon}$.
}
\end{proof}

The following lemma shows that if the direction $d^k$ in Algorithm \ref{alg:BNCG}  results from calling Algorithm \ref{pro:meo}, then the associated step length $\alpha_k$ is well-defined, and moreover, 
$\phi_\mu(x^k)-\phi_\mu(x^{k+1})$ cannot be too small.

\begin{lemma}\label{lem:NC-2nd-order-lowerbd}
Suppose that the direction $d^k$ results from calling Algorithm \ref{pro:meo} at some iteration $k$ of Algorithm \ref{alg:BNCG}. Let 
$c_{\text{nc}}$ be defined in \eqref{cnc}. Then the following statements hold.
\begin{enumerate}[{\rm (i)}]
\item The step length $\alpha_k$ is well-defined, and moreover, {the relation \eqref{ak-nc-lwbd} holds for $\alpha_k$.}
\item {$\phi_\mu(x^k)-\phi_\mu(x^{k+1})> c_{\text{nc}}\epsilon^{3/2}/64$} holds.
\end{enumerate}
\end{lemma}

\begin{proof}
Since $d^k$ results from calling Algorithm \ref{pro:meo} at some iteration $k$ of Algorithm \ref{alg:BNCG}, one has 
\begin{equation}\label{def:NC-from-meo}
d^k = -\sgn(v^TP_k^T\nabla\phi_{\mu}(x^k))\min\left\{|v^TP_k^T\nabla^2\phi_{\mu}(x^k)P_kv|,\frac{\beta}{\|Q_kv\|}\right\}v
\end{equation}
for some vector $v$ satisfying  that $\|v\|=1$ and $v^TP_k^T\nabla^2 f(x^k)P_kv\le-\sqrt{\epsilon}/2$. 
By \eqref{mu-bdd} and $\epsilon\in(0,1)$, one has that $\mu\le\epsilon/4\le\sqrt{\epsilon}/4$. 
Using these relations, and \eqref{pd-norm} with $d^k$ replaced by $v$, we obtain that
\beqa
v^TP_k^T\nabla^2\phi_{\mu}(x^k) P_kv&=&v^TP_k^T\nabla^2 f(x^k)P_kv + \mu v^TP_k^T\nabla^2 B(x^k)P_kv \le -\sqrt{\epsilon}/2 + \mu \|P_kv\|^2_{x^k} \nonumber \\[5pt] 
&\le& -\sqrt{\epsilon}/2 + \mu  \|v\|^2 {\le -\sqrt{\epsilon}/2 + \sqrt{\epsilon}/4 =  -\sqrt{\epsilon}/4.} \label{v-ineq}
\eeqa
Notice that $\|Q_kv\| \le 1$. By this, $\epsilon<1$, $\|v\|=1$, {$\beta\ge\sqrt{\epsilon}$ (see Algorithm~\ref{alg:BNCG})}, \eqref{def:NC-from-meo} and \eqref{v-ineq}, one has that 
\beq \label{dk-lwbd-nc}
\|d^k\|=\min\left\{|v^TP_k^T\nabla\phi_{\mu}(x^k)P_kv|,\frac{\beta}{\|Q_kv\|}\right\} \ge {\min\left\{\frac{\sqrt{\epsilon}}{4},\beta\right\}=\frac{\sqrt{\epsilon}}{4}.}
\eeq
In addition, one can observe from \eqref{def:NC-from-meo} that $(d^k)^TP_k^T\nabla^2\phi_{\mu}(x^k)P_kd^k$ and $v^TP_k^T\nabla^2\phi_{\mu}(x^k) P_kv$ have the same sign, which together with \eqref{v-ineq} implies that $(d^k)^TP_k^T\nabla^2\phi_{\mu}(x^k)P_kd^k<0$. By this and \eqref{def:NC-from-meo}, one has that
\[
\|d^k\|\le |v^TP_k^T\nabla^2\phi_{\mu}(x^k)P_kv| = \frac{|(d^k)^TP_k^T\nabla^2\phi_{\mu}(x^k)P_kd^k|}{\|d^k\|^2}=-\frac{(d^k)^TP_k^T\nabla^2\phi_{\mu}(x^k)P_kd^k}{\|d^k\|^2}.
\]
Hence, we obtain that $(d^k)^TP_k^T\nabla^2\phi_{\mu}(x^k)P_kd^k \le -\|d^k\|^3$. One can also observe from \eqref{def:NC-from-meo} that $(d^k)^TP_k^T\nabla\phi_{\mu}(x^k) \le 0$. The rest of the proof follows from these two relations, \eqref{dk-lwbd-nc}, and the similar arguments as used in the proof of Lemma \ref{lem:NC-step-lowerbd}.
\end{proof}

The following theorem shows that each iteration of Algorithm \ref{alg:BNCG} is well-defined, and moreover, each iterate $x^k$ generated by it is a strictly feasible point of problem \eqref{conic-prob}.  

\begin{theorem} \label{alg-well-defined}
Each iteration of Algorithm \ref{alg:BNCG} is well-defined. Moreover, each iterate $x^k$ generated by Algorithm \ref{alg:BNCG} satisfies that $x^k\in\Omega^{\rm o}$, that is, $Ax^k=b$ and $x^k\in\rmint \cK$.
\end{theorem}

\begin{proof}
From Lemma \ref{feasibility}, we know that each iterate $x^k$ generated by Algorithm \ref{alg:BNCG} satisfies that $x^k\in\cS$, which together with \eqref{level-bdd} implies that  $x^k\in\Omega^{\rm o}$, that is, $Ax^k=b$ and $x^k\in\rmint \cK$. It remains to show that each iteration of Algorithm \ref{alg:BNCG} is well-defined. To this end, suppose that $x^k$ is generated at some iteration $k$ of Algorithm \ref{alg:BNCG} and the algorithm is not terminated yet at $x^k$.  It suffices to show that the next iterate $x^{k+1}$ is successfully generated. Indeed,  since Algorithm \ref{alg:BNCG} is not terminated yet at $x^k$, then one of the following two cases must occur. As seen below,  the direction $d^k$ is successfully obtained regardless of which case occurs.

Case 1) $\min\{\|\nabla f(x^k)+A^T\lambda_k^{(1)}+\mu\nabla B(x^k)\|_{x^k}^*,\|\nabla f(x^k)+A^T\lambda_k^{(2)}+\mu\nabla B(x^{k-1})\|_{x^k}^*\}>(1-\beta)\mu$, where $\lambda_k^{(1)}$ and $\lambda_k^{(2)}$ are defined in Algorithm \ref{alg:BNCG}.  It then follows that 
\beq \label{1st-cond}
\|\nabla f(x^k)+A^T\lambda_k^{(1)}+\mu\nabla B(x^k)\|_{x^k}^*>(1-\beta)\mu.
\eeq
Claim that $P_k^T\nabla\phi_{\mu}(x^k)\neq 0$. Indeed, notice from Algorithm \ref{alg:BNCG} that $\lambda^{(1)}_k= R_k\nabla\phi_{\mu}(x^k)$. By this,  \eqref{barrier-prob}, \eqref{Mk}, \eqref{eq:relation-Pk-Rk} and \eqref{1st-cond}, one has that 
\[
\ba{l}
\|P_k^T\nabla\phi_{\mu}(x^k)\|\overset{\eqref{eq:relation-Pk-Rk}}{=}\|M_k^T(I+A^TR_k)\nabla\phi_{\mu}(x^k)\| 
\overset{\eqref{Mk}}{=}\|\nabla\phi_{\mu}(x^k)+A^TR_k\nabla\phi_{\mu}(x^k)\|_{x^k}^* \\ [4pt]
=\|\nabla f(x^k)+A^T\lambda^{(1)}_k+\mu \nabla B(x^k)\|_{x^k}^*\overset{ \eqref{1st-cond}}{>}(1-\beta)\mu,
\ea
\]
and hence $P_k^T\nabla\phi_{\mu}(x^k)\neq 0$ as claimed. Since $P_k^T\nabla\phi_{\mu}(x^k)\neq 0$, it follows from Theorem \ref{lem:capped-CG-cmplxity} that $\widehat d^k$ can be obtained from applying Algorithm \ref{alg:capped-CG} to \eqref{linsys} with $H=P_k^T\nabla^2\phi_\mu(x^k)P_k$, $\varepsilon=\sqrt{\epsilon}$, $g=P_k^T\nabla\phi_\mu(x^k)$. 
The direction $d^k$ is then obtained from $\widehat d^k$ according to \eqref{dk-nc} or \eqref{dk-sol}.

Case 2) $v^TP_k^T\nabla^2 f(x^k)P_kv<-\sqrt{\epsilon}/2$ for some unit vector $v$ returned from calling Algorithm \ref{pro:meo} with $H=P_k^T\nabla^2 f(x^k)P_k$ and  $\varepsilon=\sqrt{\epsilon}$. In this case, the direction $d^k$ is obtained from $v$ according to \eqref{dk-2nd-nc}.

In addition, it follows from Lemmas \ref{lem:SOL-step-lowerbd2}-\ref{lem:NC-2nd-order-lowerbd} that the step length $\alpha_k$ is well-defined. Hence, the next iterate $x^{k+1}$ is successfully generated by $x^{k+1}=x^k+\alpha_k P_kd^k$
\end{proof}

In the next theorem we establish iteration complexity results for Algorithm \ref{alg:BNCG}.

\begin{theorem} \label{iter-complexity}
Let
{
\begin{eqnarray}
K_1&=&\left\lceil \frac{\phi^0-\underline{\phi}}{ \min\{c_{\text{sol}},c_{\text{nc}}\}}\epsilon^{-3/2}\right\rceil+\left\lceil\frac{64(\phi^0-\underline{\phi})}{c_{\text{nc}}}\epsilon^{-3/2}\right\rceil+1,\label{K1}\\
K_2&=&\left\lceil\frac{64(\phi^0-\underline{\phi})}{c_{\text{nc}}}\epsilon^{-3/2}\right\rceil+1, \label{K2}
\end{eqnarray}
} 
where $\phi^0=f(x^0)+\bar \mu \max\{B(x^0),0\}$, $\bar \mu$ and $\underline{\phi}$ are given in Assumption \ref{main-assump},  and $c_{\text{sol}}$ and $c_{\text{nc}}$ are defined in \eqref{csol} and \eqref{cnc}, respectively. 
 Then the following statements hold.
\begin{enumerate}[{\rm (i)}]
\item The total number of calls of Algorithm \ref{pro:meo} in Algorithm \ref{alg:BNCG} is at most $K_2$.
\item The total number of calls of Algorithm \ref{alg:capped-CG} in Algorithm \ref{alg:BNCG} is at most $K_1$.
\item  Algorithm \ref{alg:BNCG} terminates in at most $K_1+K_2$ iterations. Its output $x^k$ is a deterministic $\epsilon$-first-order stationary point for some $k\le K_1+K_2 $.  Moreover, it is an  $(\epsilon,\sqrt{\epsilon})$-second-order stationary point with probability at least  
$1-\sqrt{2.75n}\delta^{\|H_k\|^{-1/2}}$, which is bounded below by $1-\sqrt{2.75n}\delta^{U_H^{-1/2}}$, where $H_k=P_k^T\nabla^2 f(x^k)P_k$ and $U_H$ is given in \eqref{Ugh}.
\end{enumerate}
\end{theorem}

\begin{proof}
(i) Suppose for contradiction that the total number of calls of Algorithm \ref{pro:meo} in Algorithm \ref{alg:BNCG} is more than $K_2$. Observe from Algorithm \ref{alg:BNCG} that each of these calls except the last one returns a sufficient negative curvature direction. Hence, these calls would totally return at least $K_2$ 
sufficient negative curvature directions. {In addition, recall from Lemma~\ref{lem:NC-2nd-order-lowerbd}(ii) that each of such directions results in a reduction on the function value of $\phi_\mu$ at least by $c_{\text{nc}}\epsilon^{3/2}/64$.}
Also, since $\mu \le \bar \mu$ and $x^k \in \cS$,  one can observe that 
\[
\phi_\mu(x^0) = f(x^0)+\mu B(x^0) \le f(x^0)+\bar \mu \max\{B(x^0),0\} = \phi^0, \qquad \phi_\mu(x^k) \ge \underline \phi \quad \forall k \in \bbK, 
\]
where $\bbK$ is given in Lemma \ref{feasibility}. Besides, notice that $\{\phi_\mu(x^k)\}_{k\in \bbK}$ is descent. Based on these observations, one would have 
\[
{
 K_2c_{\text{nc}}\epsilon^{3/2}/64} \le \sum\limits_{k\in\bbK} [\phi_\mu(x^k)-\phi_\mu(x^{k+1})]\le \phi_{\mu}(x^0) - \underline{\phi} \le \phi^0 - \underline{\phi},
\]
which contradicts with \eqref{K2}. Hence, statement (i) holds.

(ii) Suppose for contradiction that the total number of calls of Algorithm \ref{alg:capped-CG} in Algorithm \ref{alg:BNCG} is more than $K_1$. By statement (i) and Algorithm \ref{alg:BNCG}, one can observe that the total number of calls of Algorithm \ref{alg:capped-CG} that produce an iterate $x^k$ satisfying $\|\nabla f(x^k)+A^T\lambda_k^{(1)}+\mu\nabla B(x^k)\|_{x^k}^* \le(1-\beta)\mu$ or $\|\nabla f(x^k)+A^T\lambda_k^{(2)}+\mu\nabla B(x^{k-1})\|_{x^k}^* \le (1-\beta)\mu$ is at most $K_2$. 
 Using these, and Lemmas \ref{lem:SOL-step-lowerbd2} and \ref{lem:NC-step-lowerbd}, one can further observe that the total number of iterations of Algorithm~\ref{alg:BNCG}, at which Algorithm~\ref{alg:capped-CG} is called and the next iterate reduces the function value of $\phi_\mu$ at least by $\min\{c_{\rm sol},c_{\rm nc}\}\epsilon^{3/2}$, would be at least $K_1-K_2+1$. Combining these observations with the fact that $\{\phi_\mu(x^k)\}_{k\in \bbK}$ is descent, one then would have 
\[
{
(K_1-K_2+1)\min\{c_{\text{sol}},c_{\text{nc}}\}\epsilon^{3/2}}\le \sum\limits_{k\in\bbK} [\phi_\mu(x^k)-\phi_\mu(x^{k+1})] \le \phi_{\mu}(x^0) - \underline{\phi} \le \phi^0 - \underline{\phi},
\]
where $\bbK$ is given in Lemma \ref{feasibility}. This together with \eqref{K2} leads to a contradiction with \eqref{K1}.

(iii) Since either Algorithm \ref{alg:capped-CG} or Algorithm \ref{pro:meo} is called at each iteration of 
Algorithm \ref{alg:BNCG}, it follows from statements (i) and (ii) that Algorithm \ref{alg:BNCG} terminates in at most $K_1+K_2$ iterations. Suppose that Algorithm \ref{alg:BNCG} terminates at iteration $k$ for some $k \le K_1+K_2$. 
One can observe from Algorithm \ref{alg:BNCG} and Theorem \ref{rand-Lanczos} that 
\begin{equation}\label{1st-stopping}
\|\nabla f(x^k)+A^T\lambda^k+\mu\nabla B(\widetilde{x})\|_{x^k}^* \le (1-\beta)\mu
\end{equation}
for some $(\widetilde{x},\lambda^k)\in\{(x^k,\lambda_k^{(1)}),(x^{k-1},\lambda_k^{(2)})\}$, and additionally, $\lambda_{\min}(P_k^T\nabla^2 f(x^k)P_k)\ge-\sqrt{\epsilon}$ holds with a probability at least $1-\sqrt{2.75n}\delta^{\|H_k\|^{-1/2}}$, where  $H_k=P_k^T\nabla^2 f(x^k)P_k$. In addition, it follows from \eqref{xk-change} and the definition of $\widetilde{x}$  that $\|x^{k}-\widetilde{x}\|_{\widetilde{x}}\le\beta$.
By these and Lemma \ref{lem:barrier-property}(iv), one has 
$
\|\nabla f(x^k)+A^T\lambda^k+\mu\nabla B(\widetilde{x})\|_{\widetilde{x}}^*\le  \mu,
$
which yields
\[
\|(\nabla f(x^k)+A^T\lambda^k)/\mu+\nabla B(\widetilde{x})\|_{\widetilde{x}}^*\le 1.
\]
Using this and Lemma \ref{lem:barrier-property}(vi), we have $(\nabla f(x^k)+A^T\lambda^k)/\mu\in\cK^*$. Hence, \eqref{inexact-1st-order-1} holds for $(x^k,\lambda^k)$.
 We next show that \eqref{inexact-1st-order-2} also holds for $(x^k,\lambda^k)$. Indeed,  by $\widetilde{x}\in\rmint \cK$, $\|x^k-\widetilde{x}\|_{\widetilde{x}}\le\beta$, and Lemma \ref{lem:barrier-property}(i) and (iv), one has that 
\[
\|\nabla B(\widetilde{x})\|_{x^k}^*\le(1-\beta)^{-1}\|\nabla B(\widetilde{x})\|_{\widetilde{x}}^*=(1-\beta)^{-1} \sqrt{\vartheta}.
\]
By this, \eqref{1st-stopping}, and $\mu=(1-\beta)\epsilon/[2((1-\beta)^2+\sqrt{\vartheta})]$, one has that
\[
\begin{array}{l}
\|\nabla f(x^k)+A^T\lambda^k\|_{x^k}^*\le\|\nabla f(x^k)+ A^T\lambda^k +\mu \nabla B(\widetilde{x})\|_{x^k}^*+\mu\|\nabla B(\widetilde{x})\|_{x^k}^*\\[5pt]
\le (1-\beta)\mu+\frac{\mu\sqrt{\vartheta}}{1-\beta}=\frac{(1-\beta)^2+\sqrt{\vartheta}}{1-\beta}\mu={\epsilon/2} <\epsilon,
\end{array}
\]
and hence \eqref{inexact-1st-order-2} holds for $(x^k,\lambda^k)$ as desired. In addition, we know from Theorem \ref{alg-well-defined} that $Ax^k=b$ and $x^k\in\rmint\cK$. Combining these results, we conclude that $x^k$ is a deterministic $\epsilon$-first-order stationary point. Finally, recall that 
$H_k=P_k^T\nabla^2 f(x^k)P_k$, $P_k=M_kQ_k$, $\|Q_k\|=1$, and $x^k\in\cS$. In view of these, \eqref{M-norm}, \eqref{Mk} and \eqref{Ugh}, one has that
\[
\ba{lcl}
\|H_k\|&=&\|P_k^T\nabla^2 f(x^k)P_k\| \le \|M_k^T\nabla^2 f(x^k)M_k\| = \max\limits_{\|u\|_{x^k} \le 1} \|M_k^T\nabla^2 f(x^k)u\| \\ [10pt]
&=& \max\limits_{\|u\|_{x^k} \le 1} \|\nabla^2 f(x^k)u\|^*_{x^k} = \|\nabla^2 f(x^k)\|^*_{x^k} \le U_H.
\ea
\]
 Hence, we have $1-\sqrt{2.75n}\delta^{\|H_k\|^{-1/2}} \ge 1-\sqrt{2.75n}\delta^{U_H^{-1/2}}$. 
\end{proof}

\begin{remark}
From Theorem \ref{iter-complexity}, one can see that Algorithm \ref{alg:BNCG} has an iteration complexity of $\cO(\epsilon^{-3/2})$ for finding an $(\epsilon, \sqrt{\epsilon})$-second-order stationary point of problem \eqref{conic-prob}, which matches the  best known iteration complexity achieved by the methods \cite{AABHM17,BM17QC,CDHS17,CGT11ARC,CRRW21trust,CRS16trust,CRS19iN,HLY19,MR17trust,NP06cubic,OW21,ROW20,RW18} for finding an $(\epsilon, \sqrt{\epsilon})$-second-order stationary point of problem \eqref{unconstr-opt},  \eqref{boxconstr-prob} or \eqref{lpconstr-prob}.
\end{remark}

\subsection{Operation complexity} \label{oper-complex}

In this subsection we discuss operation complexity of Algorithm \ref{alg:BNCG} for solving problem \eqref{conic-prob}, which is measured by its total main operations that depend on the type of the cone $\cK$.

Notice that Algorithm~\ref{pro:meo} with $H=P_k^T\nabla^2f(x^k)P_k$ or Algorithm \ref{alg:capped-CG} with $H=P_k^T\nabla^2\phi_\mu(x^k)P_k$ is called at iteration $k$ of Algorithm \ref{alg:BNCG}.  Also, observe that the main operation of Algorithms \ref{pro:meo} and \ref{alg:capped-CG} per iteration is the product of $H$ and a vector $v$. 
In addition,  $\lambda^{(1)}_k$, $\lambda^{(2)}_k$, $\|\nabla f(x^k)+A^T\lambda_k^{(1)}+\mu\nabla B(x^k)\|_{x^k}^*$ and $\|\nabla f(x^k)+A^T\lambda_k^{(2)}+\mu\nabla B(x^{k-1})\|_{x^k}^*$  need to be computed at iteration $k$ of Algorithm \ref{alg:BNCG}.  However, one can observe that their computational cost is no higher than that of the product of $H$ and $v$. 
Also, it is clear that the computational cost of the product of $P_k^T\nabla^2f(x^k)P_k$ and $v$ 
is no higher than that of the product of $P_k^T\nabla^2\phi_\mu(x^k)P_k$ and $v$. Thus,  we only focus on the product of $H$ and $v$ with $H=P_k^T\nabla^2\phi_\mu(x^k)P_k$. We now discuss how to compute $Hv$ by utilizing the structure of $P_k^T\nabla^2\phi_\mu(x^k)P_k$. In view of \eqref{Mk}, \eqref{Qk} and \eqref{Pk}, one has 
\[
\ba{l}
Hv=P_k^T\nabla^2 \phi_\mu(x^k)P_kv= Q_kM^T_k\nabla^2 f(x^k)M_kQ_kv+\mu Q_kM^T_k\nabla^2 B(x^k)M_kQ_kv \\ [8pt]
=Q_kM^T_k\nabla^2 f(x^k)M_kQ_kv+\mu Q_kv = v^5+\mu v^1,
\ea
\]
where
\[
v^1=Q_kv, \quad v^2=M_kv^1, \quad v^3=\nabla^2 f(x^k)v^2, \quad v^4=M^T_kv^3, \quad v^5=Q_kv^4.
\]
Thus, the computation of $Hv$ is broken into that of $v^i$ for $1 \le i \le 5$. In what follows, we discuss  how to compute them, and also analyze their associated operation cost.

\bi
\item[1.] Notice that $v^1$ and $v^5$ are both a product of $Q_k$ and a vector. Let us consider computing $Q_ku$ for some vector $u$. By \eqref{Qk}, one has
\[
Q_ku = (I-M_k^TA^T(AM_kM_k^TA^T)^{-1}AM_k)u=u-M_k^TA^T(AM_kM_k^TA^T)^{-1}AM_ku=u-u^5,
\]
where
\[
u^1=M_ku, \quad u^2=Au^1,\quad u^3=(AM_kM_k^TA^T)^{-1}u^2,\quad u^4=A^Tu^3, \quad u^5=M_k^Tu^4.
\]
Observe that $AM_kM_k^TA^T$ can be computed by $N=M_k^TA^T$ and $AM_kM_k^TA^T=N^TN$. The computation of $N=M_k^TA^T$ involves $m$  products of $M_k^T$ and a vector. Once  $N$ is available, the operation cost of computing $N^TN$ is $\cO(m^2n)$. In addition, when $AM_kM_k^TA^T$ and $u^2$ are available, the operation cost of computing $u^3$ is $\cO(m^3)$. Also, once  $u^1$ and $u^3$ are available, the operation cost of computing $u^2$ and $u^4$ is $\cO(mn)$. By these observations and the fact that $m\le n$, one can see that the main operation of computing
$Q_ku$ consists of {\it $m+2$ products of $M_k$ or $M_k^T$ and a vector}, and also {\it one product of an $m\times n$ matrix and its transpose.} 
\item[2.] Based on the above discussion, one can observe that the computation of $v^1$, $v^2$, $v^4$, and $v^5$ involves $2m+6$ products of $M_k$ or $M_k^T$ and a vector in total.
We now discuss how to compute the product of $M_k$ or $M_k^T$ and a vector. Observe from \eqref{Mk} that $\nabla^2 B(x^k)=M_k^{-T} M_k^{-1}$. 
Thus, $M_k^{-T}$ can be obtained as the Cholesky factor of $\nabla^2 B(x^k)$, which is computed only once in each iteration of Algorithm~\ref{alg:BNCG}. 
Once $M_k^{-T}$ is available, the product of $M_k$ or $M_k^T$ and a vector can be computed by applying backward or forward substitution  to a linear system with coefficient matrix $M_k^{-1}$ or $M_k^{-T}$.
\item[3.] Once $v^2$ is available, the computation of $v^3$ only involves the product of $\nabla^2 f(x^k)$ and $v^2$.
\ei

Consequently, once the Cholesky factor $M_k^{-T}$ of $\nabla^2 B(x^k)$ is computed in each iteration of Algorithm~\ref{alg:BNCG}, the main computation of $Hv$ consists of:
\bi
\item $2m+6$ backward or forward substitutions to a linear system with coefficient matrix $M_k^{-1}$ or $M_k^{-T}$; 
\item one product of an $m\times n$ matrix and its transpose; 
\item one product of $\nabla^2 f(x^k)$ and a vector.
\ei

When $\cK$ is the nonnegative orthant, its associated LHSC is $B(x)=-\sum^n_{i=1}\ln x_i$ and $\nabla^2 B(x^k)$ is a diagonal matrix. The operation cost of the Cholesky factorization of $\nabla^2 B(x^k)$ is $\cO(n)$. In addition, the operation cost of $2m+6$ backward or forward substitutions to a linear system with coefficient matrix $M_k^{-1}$ or $M_k^{-T}$ is $\cO(mn)$. Thus, the main operation of computing $Hv$ consists of one product of an $m\times n$ matrix and its transpose, and one product of $\nabla^2 f(x^k)$ and a vector.

When $\cK$ is a general cone, such as a second-order or semidefinite cone, the operation cost of the Cholesky factorization of $\nabla^2 B(x^k)$ (including the evaluation of $\nabla^2 B(x^k)$) is typically at least $\cO(n^3)$. In addition, the operation cost of $2m+6$ backward or forward substitutions to a linear system with coefficient matrix $M_k^{-1}$ or $M_k^{-T}$ is $\cO(mn^2)$.

The above discussion and Theorems~\ref{lem:capped-CG-cmplxity}(ii), \ref{rand-Lanczos} and \ref{iter-complexity} lead to the following operation complexity results for Algorithm \ref{alg:BNCG}, which are represented by its total main operations that depend on the type of the cone $\cK$.

\begin{theorem}\label{thm:operation-cmplxty}
Let $K_1$ and $K_2$ be given in \eqref{K1} and \eqref{K2}, respectively, and let 
\[
\bar N=\widetilde{\cO}\left(\min\{n,\epsilon^{-1/4}\} K_1+\min\left\{n,1+\left\lceil \epsilon^{-1/4}\ln\delta^{-1}\right\rceil\right\}K_2\right).
\]
 Then the following statements hold.
\begin{enumerate}[{\rm (i)}]
\item When $\cK$ is the nonnegative orthant, the total main operations of Algorithm \ref{alg:BNCG} consist of  $\bar N$ Hessian-vector products of $f$ and $\bar N$ products of an $m\times n$ matrix and its transpose. 
\item When $\cK$ is a general cone, the total main operations of Algorithm \ref{alg:BNCG} consist of $K_1+K_2$ Cholesky factorizations of the Hessian of $B$, 
$\bar N$ Hessian-vector products of $f$, 
and $(2m+6)\bar{N}$ backward or forward substitutions to a linear system with a lower or upper triangular coefficient matrix.   
\end{enumerate}
\end{theorem}

\begin{remark}
Recall from Theorem~\ref{iter-complexity} that $K_1=\cO(\epsilon^{-3/2})$ and $K_2=\cO(\epsilon^{-3/2})$. In view of these and Theorem~\ref{thm:operation-cmplxty}, we observe that
\begin{enumerate}[{\rm (i)}]
\item when $\cK$ is the nonnegative orthant, Algorithm~\ref{alg:BNCG} achieves an operation complexity of $\widetilde{\cO}(\epsilon^{-3/2}\min\{n,\epsilon^{-1/4}\})$, measured by the amount of main operations consisting of Hessian-vector products of $f$ and also products of an $m\times n$ matrix and its transpose, for finding an $(\epsilon,\sqrt{\epsilon})$-second-order stationary point of \eqref{conic-prob} with high probability;

\item when $\cK$ is a general cone, Algorithm~\ref{alg:BNCG} requires at most $\cO(\epsilon^{-3/2})$ Cholesky factorizations of the Hessian of $B$ and $\widetilde{\cO}(\epsilon^{-3/2}\min\{n,\epsilon^{-1/4}\})$ other fundamental operations, consisting of Hessian-vector products of $f$ and backward or forward substitutions to a lower or upper triangular linear system, for finding an $(\epsilon,\sqrt{\epsilon})$-second-order stationary point of \eqref{conic-prob} with high probability.
\end{enumerate}
In addition, when $A=0$, $b=0$ and $\cK$ is the nonnegative orthant, the aforementioned operation complexity for Algorithm~\ref{alg:BNCG} matches the best known ones of second-order methods for finding an $(\epsilon,\sqrt{\epsilon})$-second-order stationary point of problem~\eqref{unconstr-opt} or \eqref{boxconstr-prob} with high probability (e.g., see \cite{CRRW21trust,OW21,ROW20}). 
\end{remark}

\section*{Appendix}

\appendix

\section{A capped conjugate gradient method} \label{appendix:capped-CG}

In this part we present a capped CG method proposed in \cite{ROW20} for finding either an approximate solution of \eqref{linsys} or a negative curvature direction of the matrix $H$, which has been discussed in  Subsection \ref{sec:capped-cg-meo}. The detailed motivation and explanation of this method can be found in  \cite{ROW20}.

\begin{algorithm}[h]
\caption{A capped conjugate gradient method}
\label{alg:capped-CG}
{\small
\begin{algorithmic}
\State \noindent\textit{Input}: Symmetric matrix $H\in\bR^{n\times n}$, vector $g\neq0$, damping parameter $\varepsilon\in(0,1)$, desired relative accuracy $\zeta\in(0,1)$.
\State \textit{Optional input:} scalar $U\ge0$ such that $\|H\|\le U$ (set to $0$ if not provided).
\State \textit{Output:} d$\_$type, $\hd$.
\State \textit{Secondary output:} final values of $U,\kappa,\widehat{\zeta},\tau,$ and $T$.
\State Set 
\begin{equation*}
\bar{H}:=H+2\varepsilon I,\quad \kappa:=\frac{U+2\varepsilon}{\varepsilon},\quad\widehat{\zeta}:=\frac{\zeta}{3\kappa},\quad\tau:=\frac{\sqrt{\kappa}}{\sqrt{\kappa}+1},\quad T:=\frac{4\kappa^4}{(1-\sqrt{\tau})^2},
\end{equation*}
$y^0\leftarrow 0,r^0\leftarrow g,p^0\leftarrow -g, j\leftarrow 0$.
\If {$(p^0)^T \bar{H}p^0<\varepsilon\|p^0\|^2$}
\State Set $\hd=p^0$ and terminate with d$\_$type = NC;
\ElsIf {\ $\|Hp^0\|>U\|p^0\|\ $}
\State Set $U\leftarrow\|Hp^0\|/\|p^0\|$ and update $\kappa,\widehat{\zeta},\tau, T$ accordingly;
\EndIf
\While{TRUE}
\State $\alpha_j\leftarrow (r^j)^T r^j/(p^j)^T\bar{H}p^j$; \{Begin Standard CG Operations\}
\State $y^{j+1}\leftarrow y^j+\alpha_jp^j$;
\State $r^{j+1}\leftarrow r^j+\alpha_j\bar{H}p^j$;
\State $\beta_{j+1}\leftarrow\|r^{j+1}\|^2/\|r^j\|^2$;
\State $p^{j+1}\leftarrow-r^{j+1}+\beta_{j+1}p^j$; \{End Standard CG Operations\}
\State $j\leftarrow j+1$;
\If {$\|Hp^j\|>U\|p^j\|$}
\State Set $U\leftarrow\|Hp^j\|/\|p^j\|$ and update $\kappa,\widehat{\zeta},\tau,T$ accordingly;
\EndIf
\If {\ $\|Hy^j\|>U\|y^j\|\ $}
\State Set $U\leftarrow\|Hy^j\|/\|y^j\|$ and update $\kappa,\widehat{\zeta},\tau,T$ accordingly;
\EndIf
\If {\ $\|Hr^j\|>U\|r^j\|\ $}
\State Set $U\leftarrow\|Hr^j\|/\|r^j\|$ and update $\kappa,\widehat{\zeta},\tau,T$ accordingly;
\EndIf
\If {$(y^j)^T\bar{H}y^j<\varepsilon\|y^j\|^2$}
\State Set $\hd\leftarrow y^j$ and terminate with d$\_$type = NC;
\ElsIf {\ $\|r^j\|\le\widehat{\zeta}\|r^0\|$}
\State Set $\hd\leftarrow y^j$ and terminate with d$\_$type = SOL;
\ElsIf{\ $(p^j)^T\bar{H}p^j<\varepsilon\|p^j\|^2$}
\State Set $\hd\leftarrow p^j$ and terminate with d$\_$type = NC;  
\ElsIf {\ $\|r^j\|>\sqrt{T}\tau^{j/2}\|r^0\| $}
\State Compute $\alpha_j, y^{j+1}$ as in the main loop above;
\State Find $i\in\{0,\ldots,j-1\}$ such that
\[
(y^{j+1}-y^i)^T\bar{H}(y^{j+1}-y^i)<\varepsilon\|y^{j+1}-y^i\|^2;
\]
\State Set $\hd\leftarrow y^{j+1}-y^i$ and terminate with d$\_$type = NC;
\EndIf
\EndWhile
\end{algorithmic}
}
\end{algorithm}

\end{document}